\documentclass[a4paper, leqno, 11pt]{amsart}

\usepackage[T1]{fontenc}	
\usepackage{graphicx}
\usepackage{amssymb,centernot}
\usepackage{upgreek}
\usepackage{mathrsfs}
\usepackage[usenames,dvipsnames]{xcolor}
\usepackage[inline,shortlabels]{enumitem}
\usepackage[sort,numbers]{natbib}	
\usepackage{verbatim}								%per commentare cose lunghe
\usepackage{dsfont}									%doublestroke \mathds
\usepackage{pdflscape}
\usepackage{afterpage}

\usepackage{anyfontsize}
\usepackage[
						colorlinks,				%link a colori invece che riquadrati
						breaklinks,unicode,
						hypertexnames=false,
						citecolor=OliveGreen,
						linkcolor=Maroon
						]{hyperref} %hyperref deve stare per ultimo

\usepackage{multirow}

\usepackage{xypic}
\usepackage{mathtools}
\usepackage{xspace}

\usepackage[a4paper,top=3.2cm,bottom=2.7cm,left=3cm,right=3cm, bindingoffset=0mm]{geometry}
\usepackage{booktabs}

\usepackage{tikz}

\usetikzlibrary{calc, positioning, shapes.geometric, patterns, cd}

\newlength{\hatchspread}
\newlength{\hatchthickness}
\newlength{\hatchshift}
\newcommand{\hatchcolor}{}
\tikzset{hatchspread/.code={\setlength{\hatchspread}{#1}},
         hatchthickness/.code={\setlength{\hatchthickness}{#1}},
         hatchshift/.code={\setlength{\hatchshift}{#1}},% must be >= 0
         hatchcolor/.code={\renewcommand{\hatchcolor}{#1}}}
\tikzset{hatchspread=3pt,
         hatchthickness=0.4pt,
         hatchshift=0pt,% must be >= 0
         hatchcolor=black}
\pgfdeclarepatternformonly[\hatchspread,\hatchthickness,\hatchshift,\hatchcolor]% variables
   {custom north west lines}% name
   {\pgfqpoint{\dimexpr-2\hatchthickness}{\dimexpr-2\hatchthickness}}% lower left corner
   {\pgfqpoint{\dimexpr\hatchspread+2\hatchthickness}{\dimexpr\hatchspread+2\hatchthickness}}% upper right corner
   {\pgfqpoint{\dimexpr\hatchspread}{\dimexpr\hatchspread}}% tile size
   {% shape description
    \pgfsetlinewidth{\hatchthickness}
    \pgfpathmoveto{\pgfqpoint{0pt}{\dimexpr\hatchspread+\hatchshift}}
    \pgfpathlineto{\pgfqpoint{\dimexpr\hatchspread+0.15pt+\hatchshift}{-0.15pt}}
    \ifdim \hatchshift > 0pt
      \pgfpathmoveto{\pgfqpoint{0pt}{\hatchshift}}
      \pgfpathlineto{\pgfqpoint{\dimexpr0.15pt+\hatchshift}{-0.15pt}}
    \fi
    \pgfsetstrokecolor{\hatchcolor}
    \pgfusepath{stroke}
   }

\pgfdeclarepatternformonly[\hatchspread,\hatchthickness,\hatchshift,\hatchcolor]% variables
   {custom north east lines}% name
   {\pgfqpoint{\dimexpr-2\hatchthickness}{\dimexpr-2\hatchthickness}}% lower left corner
   {\pgfqpoint{\dimexpr\hatchspread+2\hatchthickness}{\dimexpr\hatchspread+2\hatchthickness}}% upper right corner
   {\pgfqpoint{\dimexpr\hatchspread}{\dimexpr\hatchspread}}% tile size
   {% shape description
    \pgfsetlinewidth{\hatchthickness}
    \pgfpathmoveto{\pgfqpoint{\dimexpr\hatchshift-0.15pt}{-0.15pt}}
    \pgfpathlineto{\pgfqpoint{\dimexpr\hatchspread+0.15pt}{\dimexpr\hatchspread-\hatchshift+0.15pt}}
    \ifdim \hatchshift > 0pt
      \pgfpathmoveto{\pgfqpoint{-0.15pt}{\dimexpr\hatchspread-\hatchshift-0.15pt}}
      \pgfpathlineto{\pgfqpoint{\dimexpr\hatchshift+0.15pt}{\dimexpr\hatchspread+0.15pt}}
    \fi
    \pgfsetstrokecolor{\hatchcolor}
    \pgfusepath{stroke}
   }

\makeatletter
\newcommand*{\centerfloat}{%
  \parindent \z@
  \leftskip \z@ \@plus 1fil \@minus \textwidth
  \rightskip\leftskip
  \parfillskip \z@skip}
\makeatother

\usepackage{xparse}

\ExplSyntaxOn
\NewDocumentCommand{\makeabbrev}{mmm}
 {
  \yoruk_makeabbrev:nnn { #1 } { #2 } { #3 }
 }

\cs_new_protected:Npn \yoruk_makeabbrev:nnn #1 #2 #3
 {
  \clist_map_inline:nn { #3 }
   {
    \cs_new_protected:cpn { #2 } { #1 { ##1 } }
   }
 }
 \ExplSyntaxOff

\makeabbrev{\textbf}{tbf#1}{a,b,c,d,e,f,g,h,i,j,k,l,m,n,o,p,q,r,s,t,u,v,w,x,y,z,A,B,C,D,E,F,G,H,I,J,K,L,M,N,O,P,Q,R,S,T,U,V,W,X,Y,Z}

\makeabbrev{\textbf}{bf#1}{a,b,c,d,e,f,g,h,i,j,k,l,m,n,o,p,q,r,s,t,u,v,w,x,y,z,A,B,C,D,E,F,G,H,I,J,K,L,M,N,O,P,Q,R,S,T,U,V,W,X,Y,Z}

\makeabbrev{\textsf}{tsf#1}{a,b,c,d,e,f,g,h,i,j,k,l,m,n,o,p,q,r,s,t,u,v,w,x,y,z,A,B,C,D,E,F,G,H,I,J,K,L,M,N,O,P,Q,R,S,T,U,V,W,X,Y,Z}

\makeabbrev{\mathsf}{mss#1}{a,b,c,d,e,f,g,h,i,j,k,l,m,n,o,p,q,r,s,t,u,v,w,x,y,z,A,B,C,D,E,F,G,H,I,J,K,L,M,N,O,P,Q,R,S,T,U,V,W,X,Y,Z}

\makeabbrev{\mathfrak}{mf#1}{a,b,c,d,e,f,g,h,i,j,k,l,m,n,o,p,q,r,s,t,u,v,w,x,y,z,A,B,C,D,E,F,G,H,I,J,K,L,M,N,O,P,Q,R,S,T,U,V,W,X,Y,Z}

\makeabbrev{\mathrm}{mrm#1}{a,b,c,d,e,f,g,h,i,j,k,l,m,n,o,p,q,r,s,t,u,v,w,x,y,z,A,B,C,D,E,F,G,H,I,J,K,L,M,N,O,P,Q,R,S,T,U,V,W,X,Y,Z}

\makeabbrev{\mathbf}{mbf#1}{a,b,c,d,e,f,g,h,i,j,k,l,m,n,o,p,q,r,s,t,u,v,w,x,y,z,A,B,C,D,E,F,G,H,I,J,K,L,M,N,O,P,Q,R,S,T,U,V,W,X,Y,Z}

\makeabbrev{\mathcal}{mc#1}{A,B,C,D,E,F,G,H,I,J,K,L,M,N,O,P,Q,R,S,T,U,V,W,X,Y,Z}

\makeabbrev{\mathbb}{mbb#1}{A,B,C,D,E,F,G,H,I,J,K,L,M,N,O,P,Q,R,S,T,U,V,W,X,Y,Z}

\makeabbrev{\mathscr}{ms#1}{A,B,C,D,E,F,G,H,I,J,K,L,M,N,O,P,Q,R,S,T,U,V,W,X,Y,Z}

\makeabbrev{\mathrm}{#1}{
Id,id,ran,rk,diag,stab,ann,conv,pr,ev,tr,End,Hom,sgn,im,op,can,fin,ext,red,tot,
rot,usc,lsc,Lip,LocLip,lip,bSymLip,osc,AC,loc,uloc,spec,coz,z,ul,
supp,Opt,Adm,Cpl,Geo,GeoSel,GeoOpt,GeoAdm,GeoCpl,reg,
bd,co,Ric,Exp,dExp,dist,seg,Seg,cut,fcut,Cut,SDiff,Iso,Isom,diam,cl,Homeo,Diff,Der,vol,dvol,inj,relint, Graph, sub,codim,
var,law,Var,Poi,Gam,pa,so,iso,fs,inv,pqi,mix,
TestF,
}

\makeabbrev{\mathsf}{#1}{DP,CD,BE,MCP,Ent,wMTW,MTW,RCD,ncRCD,QCD,EVI,Irr,IH,SC,wFe,VA,UP,Curv,Alex,CAT}

\newcommand{\T}{\tau} %TOPOLOGY
\renewcommand{\div}{\mathrm{div}}

\newcommand{\Ch}[1][]{\mathsf{Ch}_{#1}}

\newcommand{\e}{\epsilon}

\let\epsilon\varepsilon

\let\temp\phi
\let\phi\varphi
\let\varphi\temp

\newcommand{\diff}{\mathop{}\!\mathrm{d}}
\DeclareSymbolFont{symbolsC}{U}{pxsyc}{m}{n}
\SetSymbolFont{symbolsC}{bold}{U}{pxsyc}{bx}{n}
\DeclareFontSubstitution{U}{pxsyc}{m}{n}
\DeclareMathSymbol{\medcirc}{\mathbin}{symbolsC}{7}
\DeclareSymbolFont{symbolsZ}{OMS}{pxsy}{m}{n}
\SetSymbolFont{symbolsZ}{bold}{OMS}{pxsy}{bx}{n}
\DeclareFontSubstitution{OMS}{pxsy}{m}{n}

\DeclareMathOperator*{\esssup}{ess\, sup}
\DeclareMathOperator*{\essinf}{ess\, inf}

\newcommand{\N}{{\mathbb N}}
\newcommand{\R}{{\mathbb R}}

\allowdisplaybreaks

\usetikzlibrary{shapes.misc}
\tikzset{cross/.style={cross out, draw=black, minimum size=2*(#1-\pgflinewidth), inner sep=0pt, outer sep=0pt},
cross/.default={4pt}}

\newcommand{\comma}{\,\,\mathrm{,}\;\,}

\newcommand{\fstop}{\,\,\mathrm{.}}

\usepackage{scrextend}						%KOMA: confligge quindi deve stare qui

\newcommand{\E}{\mathcal E}

\renewcommand{\1}{\mathbf 1}

\renewcommand{\d}{{\sf d}}
\newcommand{\m}{{\sf m}}
\newcommand{\M}{{\sf M}}
\newcommand{\X}{{\sf X}}
\newcommand{\Dif}{{\rm D}}
\newcommand{\dif}{{\mathbf d}}
\renewcommand{\T}{{\sf T}}
\renewcommand{\P}{{\sf P}}
\renewcommand{\L}{{\mathbb L}}

\numberwithin{equation}{section}
\theoremstyle{plain}
\newtheorem{thm}{Theorem}[section]
\newtheorem*{thm*}{Theorem}
\newtheorem*{mthm*}{Main Theorem}

\newtheorem{prop}[thm]{Proposition}%[section]
\newtheorem{lem}[thm]{Lemma}%[section]
\newtheorem{cor}[thm]{Corollary}%[section]
\newtheorem*{cor*}{Corollary}

\theoremstyle{definition}
\newtheorem*{defs*}{Definition}%[section]
\theoremstyle{remark}

\newtheorem{rem}[thm]{\bf Remark}%[section]
\renewcommand{\paragraph}[1]{\medskip\emph{#1}.\quad}

\date{\today}
\begin{document}
\title[Integral Varadhan formula for non-linear heat flow]{Integral Varadhan formula for non-linear heat flow}

\author[S.~Ohta]{Shin-ichi Ohta}
\address{Department of Mathematics, Osaka University, Osaka 560-0043, Japan, and RIKEN Center for Advanced Intelligence Project (AIP), 1-4-1 Nihonbashi, Tokyo 103-0027, Japan}
\email{s.ohta@math.sci.osaka-u.ac.jp}

\author[K.~Suzuki]{Kohei Suzuki}
\address{Department of Mathematical Science, Durham University, DH13LE, Durham, United Kingdom/Theoretical Sciences Visiting Program, Okinawa Institute of Science and Technology Graduate University, Onna, 904-0495, Japan}
\email{kohei.suzuki@durham.ac.uk}
\thanks{The first author is supported in part by JSPS Grant-in-Aid for Scientific Research (KAKENHI) 19H01786, 22H04942. The second author gratefully acknowledges funding by the Alexander von Humboldt Stiftung as well as the Theoretical Sciences Visiting Program (TSVP) at the Okinawa
Institute of Science and Technology Graduate University in Japan.}

%\keywords{\vspace{2mm}Ergodicity, tail-triviality, optimal transport, Sobolev-to-Lipschitz, rigidity}
%
%\subjclass[2010]{37A30, 31C25, 30L99, 70F45, 60G55}

\maketitle

\begin{abstract}
We prove the integral Varadhan short-time formula for non-linear heat flow on measured Finsler manifolds. 
To the best of the authors' knowledge, this is the first result establishing a Varadhan-type formula 
for non-linear semigroups. 
We do not assume the reversibility of the metric, thus the distance function can be asymmetric. 
In this generality, we reveal that the probabilistic interpretation is well-suited for our formula; 
the probability that a particle starting from a set $A$ can be found in another set $B$ 
describes the distance from $A$ to $B$. 
One side of the estimates (the upper bound of the probability) is also established  
in the nonsmooth setting of infinitesimally strictly convex metric measure spaces. 
%satisfying the local Sobolev-to-Lipschitz property. 
\end{abstract}

\section{Introduction}%%%%%
%%%%%%%%%

\paragraph{Aim}%%%%%
The main aim of this article is to draw more attention to (geometric) analysis of \emph{non-linear heat flow}. 
The linear theory had been developed intensively and extensively in connection with two powerful theories: 
\emph{Dirichlet forms} related to probability theory and the \emph{$\Gamma$-calculus} 
\`a la Bakry--\'Emery related to differential geometry as well as geometric analysis. 
A non-linear analogue to the $\Gamma$-calculus has been investigated on Finsler manifolds 
(of Ricci curvature bounded below in an appropriate way) 
by the first author and Sturm \cite{Oht17b,Oht17a,Oht21,Oht22,OS14}. 
Then it is natural to expect a more general theory of non-linear heat semigroups 
as a non-linear counterpart to the theory of Dirichlet forms, however, 
there is surprisingly no result in such a direction. 
In this article, to motivate further studies of non-linear heat semigroups, 
we establish the integral Varadhan short-time formula for non-linear heat flow on Finsler manifolds. 
%\purple{{\bf delete} \footnotesize A large part of our discussion can be generalised to metric measure spaces under mild assumptions.} 

\paragraph{Background}%%%%%
On a complete Riemannian manifold $(\M,g)$ with the Riemannian distance $\d$, 
let $p_t(x,y)$ be the heat kernel density, i.e., the minimal fundamental solution to the heat equation 
$\partial_t u= \frac{1}{2}\Delta u$. 
From the probabilistic viewpoint, $p_t(x,y)$ is the density function of the transition probability 
of the Brownian motion in $\M$. 
The \emph{Varadhan short-time formula}~\cite{Var67} states that 
the short-time behaviour of $p_t(x,y)$ is governed by $\d(x,y)$ in the following way: 
\begin{equation}\label{VAR}
\lim_{t \downarrow 0} t\log p_t(x,y) = -\frac{1}{2}\d(x,y)^2 \fstop
\end{equation}
The formula \eqref{VAR}, linking geometry, analysis, and probability, 
has been studied in various settings including complete connected Riemannian manifolds~\cite{Var67}, 
Lipschitz manifolds~\cite{Nor97}, degenerate diffusions on Euclidean spaces~\cite{CarKusStr87}, 
sub-Riemannian manifolds~\cite{Lea87a, Lea87b, Ben88, BenLea91, BaiNor18}, 
and metric measure spaces satisfying the quasi Riemannian curvature-dimension condition~\cite{LzDSSuz22}.
%sub-Riemannian manifolds ??, fractals ??, metric measure spaces with curvature bound ??. 

For spaces not admitting the heat kernel density~$p_t(x,y)$, 
the formula \eqref{VAR} has been generalised as 
\begin{equation} \label{e:FV}
\lim_{t \downarrow 0}t\log \P_t(A, B) = -\frac{1}{2}\bar\d_\m(A, B)^2 
\end{equation}
for $A,B \subset \M$ with $0<\m(A), \m(B)<\infty$, where $\m$ is the reference measure on $\M$,
\[ \P_t(A, B):= \int_A \T_t\1_B \diff \m \]
with the $L^2$-heat semigroup~$(\T_t)_{t \ge 0}$, 
and $\bar\d_\m(A, B)$ is a suitably defined distance-like function. 
In the case of \emph{linear} heat semigroups, \eqref{e:FV} has been established 
in a general setting of local Dirichlet spaces \cite{HinRam03, AriHin05, HinMat18}, 
where $\bar\mssd_\m(A, B)$ is induced by a local Dirichlet form. 
If a local Dirichlet space admits a distance function in the domain of the Dirichlet form 
in a compatible way in the sense of the Rademacher-type property and the Sobolev-to-Lipschitz property, 
the set function~$\bar\mssd_\m(A, B)$ is indeed identified with the distance between $A$ and $B$; 
see~\cite{LzDSSuz20} for details. 
We refer the readers to the following articles for particular spaces: 
the Wiener space and path/loop groups \cite{Fan94, FanZha99, AidKaw01, AidZha02, HinRam03}, 
the configuration space \cite{Zha01, LzDSSuz22a} and the Wasserstein space \cite{vReStu09}.
%\purple{{\bf Delete} RCD spaces \cite{LzDSSuz20}, and~\cite{GigTamTre22} for a large deviation type result.}

\paragraph{Main results}%%%%%
%By extending the arguments in \cite{HinRam03, AriHin05}, 
We shall generalise \eqref{e:FV} to \emph{non-linear} heat flow $(\T_t)_{t \ge 0}$ 
on a Finsler manifold $(\M,F)$ equipped with a measure $\m$. 
We do not assume the \emph{reversibility} of $F$ (i.e., $F(-v) \neq F(v)$ is allowed), 
thereby the distance function $\d$ can be asymmetric. 
In this case, $\bar\d_\m(A,B)$ is defined as 
\begin{equation}\label{eq:d_m}
\bar\d_\m(A, B) :=\sup_{f \in \L} \Bigl\{ \essinf_{x \in A}f(x)-\esssup_{y \in B}f(y) \Bigr\} 
\end{equation}
for measurable sets $A,B \subset \M$ with $0< \m(A),\m(B) <\infty$, where 
\begin{equation}\label{eq:L}
\L:=\{f \in H^1_{\loc}(\M) \cap L^\infty(\M): F^*(-\dif f) \le 1\ \text{a.e.} \} \fstop 
\end{equation}
We remark that $\bar\d_\m(A,B)<\infty$. 
The condition $F^*(-\dif f) \le 1$ roughly means that $-f$ is $1$-Lipschitz, 
and one can regard that $\bar\d_\m(A,B)$ represents the distance \emph{from $A$ to $B$}. 
We refer to Subsection~\ref{ssc:Finsler} for precise definitions and notations in Finsler geometry. 
We also set 
\[ \d(A, B):= \inf_{x \in A,\, y \in B} \d(x, y) \fstop \]
Due to Lemma~\ref{lm:CDD} proven later, for open sets $A, B \subset \M$, we have 
\[ \bar\mssd_{\mssm}(A, B)=\d(A, B)\fstop \]

\begin{thm}\label{t:m1}
Let $(\M,F)$ be a complete $C^{\infty}$-Finsler manifold 
equipped with a $C^{\infty}$-measure $\m$ on $\M$ with $\m(\M)<\infty$. 
Assume that the uniform convexity and smoothness constants are finite. 
Then, for any measurable sets $A,B \subset \M$ with $0< \m(A),\m(B) <\infty$, we have 
\begin{equation}\label{eq:dAB}
\lim_{t \downarrow 0} t\log\P_t(A,B) =-\frac{1}{2} \bar\d_\m(A, B)^2 \fstop 
\end{equation}
In particular, for any open sets $A,B \subset \M$, we have
\[
\lim_{t \downarrow 0} t\log\P_t(A,B) =-\frac{1}{2} \d(A,B)^2 \fstop 
\]
\end{thm}

To the best of our knowledge, 
Theorem~\ref{t:m1} is the first result establishing the integral Varadhan formula for non-linear semigroups. 
It is unclear whether the pointwise Varadhan estimate \eqref{VAR} can be properly formulated in our setting, 
since there is no concept of heat kernel density $p_t(x, y)$ for non-linear heat semigroups.
The assumption $\mssm(\M)<\infty$ is used only for the lower estimate of \eqref{eq:dAB}.
Though this assumption does not seem essential, we were not able to drop it due to the technicality of the linearised heat semigroup; see Subsection~\ref{ssc:outro} for more details.
The finiteness of the uniform convexity and smoothness constants is imposed also for a technical reason of constructing a linearised heat semigroup (see Lemma~\ref{l:215}), and is satisfied by, e.g., compact Finsler manifolds and Randers spaces $(\M,F)$, $F(v)=\sqrt{g(v,v)}+\beta(v)$, such that $g$ is a Riemannian metric and $\beta$ is a one-form on $\M$ with $|\beta|_g \le c<1$ for some $c<1$.

The asymmetry of the distance function $\d$ reveals the probabilistic nature of our Varadhan formula, 
which is not apparent in the symmetric setting. 
From the analytic (PDE) point of view, 
the semigroup $\T_t \1_B$ in $\P_t(A, B)= \int_A \T_t\1_B \diff \m$ 
could be regarded as describing the heat propagation from $B$, 
thereby the appearance of the distance $\d(A,B)$ from $A$ to $B$ may be counter-intuitive. 
From the probabilistic viewpoint (which is a dual perspective to the PDE one), however, $\d(A,B)$ is natural 
since $\T_t \1_B(x)$ represents the probability that a Brownian motion starting from $x$ lives in $B$ at time $t$ 
(in the Riemannian setting, to be precise; the existence of the Brownian motion is unknown in the Finsler case). 
%\footnote{The expression of the paragraph has been slightly changed. The previous version was commented out. I added (in the Riemannian setting) because the existence of the Brownian motion is open in the Finsler case.}

%The asymmetry of the distance function $\d$ reveals the probabilistic nature of our Varadhan formula, 
%which is not apparent in the symmetric setting. 
%In $\P_t(A, B)= \int_A \T_t\1_B \diff \m$, from the analytic (PDE) point of view, 
%$\T_t \1_B$ could be regarded as describing the heat propagation from $B$, 
%thereby the appearance of the distance $\d(A,B)$ from $A$ to $B$ may be counter-intuitive. 
%On the other hand, from the probabilistic viewpoint (which is dual to the PDE one), 
%$[\T_t \1_B](x)$ represents the probability that a Brownian motion starting from $x$ lives in $B$ at time $t$. 
%Therefore, it is natural to obtain $\d(A,B)$ in Theorem~\ref{t:m1}. 

The above observation should be compared with the fact that 
heat flow is regarded as the gradient flow of the relative entropy in the $L^2$-Wasserstein space 
with respect to the reverse Finsler structure $\overline{F}(v)=F(-v)$ (see Remark~\ref{rm:GF}). 
Since $\d(A,B)$ coincides with the distance from $B$ to $A$ with respect to $\overline{F}$, 
the analytic point of view seems consistent with $\overline{F}$. 

The upper estimate in Theorem~\ref{t:m1} is more flexible than the lower estimate, 
and can be generalised to the nonsmooth setting as follows, 
thanks to differential calculus developed in \cite{AmbGigSav14, Gig18}. 
We remark that reversible Finsler manifolds (satisfying $F(-v)=F(v)$) also fall into this framework. 
%\footnote{Is this remark necessary? The reversible case is already contained in the first statement.}

\begin{thm}\label{t:m2}
Let $(\X, \d, \m)$ be an infinitesimally strictly convex metric measure space.
Then, for any measurable sets $A,B \subset \X$ with $0< \m(A),\m(B) <\infty$, we have 
\[
\limsup_{t \downarrow 0}t\log \P_t(A, B) \le -\frac{1}{2}\mssd(A, B)^2  \fstop
\]
%Furthermore, if the local Sobolev-to-Lipschitz property holds, then we have 
%\[ \bar\d_\m(A, B) = \purple{\d}(A, B) \]
%for any open sets $A, B \subset \X$ with $0< \m(A),\m(B) <\infty$. 
\end{thm}

Note that sets $A, B \subset \X$ in Theorem~\ref{t:m2} need not be open.
If, in addition to the infinitesimal strict convexity, the Sobolev space~$W^{1,2}(\X, \d, \mssm)$ (that is not necessarily a Hilbert space in this generality) is reflexive, then, for $A,B \subset \X$ with $0< \m(A),\m(B) <\infty$,
\[
\limsup_{t \downarrow 0}t\log \P_t(A, B) \le -\frac{1}{2}\bar\mssd_\mssm(A, B)^2  \comma
\]
where $\bar\mssd_\mssm(A, B)$ is defined in Subsection~\ref{subsec:MM}.
If, furthermore, $(\X, \d, \mssm)$ possesses the local Sobolev-to-Lipschitz property (see Subsection~\ref{subsec:MM}), then we have
$\bar\mssd_\mssm(A, B)=\mssd(A, B)$ for open sets $A, B \subset \X$. 

On the one hand, both the infinitesimal strict convexity and the reflexivity of the Sobolev space hold, e.g., for RCD spaces.
On the other hand, Theorem~\ref{t:m2} (more precisely, Proposition~\ref{pr:U}) can be applied to some spaces without infinitesimal strict convexity by approximation; see Remark~\ref{rm:stability}, where we discuss the space $(\R^n, \|\cdot\|_p, \mssm)$ with the $\ell_p$ norm $\|\cdot\|_p$ for $1 \le p \le \infty$ and the Lebesgue measure $\mssm$.

Compared to the existing literature, there are two difficulties that need to be addressed: one arises from the fact that 
the energy form is not bilinear, due to the lack of the Leibniz rule for the gradient operator;
another difficulty arises  due to  the lack of symmetry: 
$\int_\M u_1 \Delta u_2 \diff\m \neq \int_\M u_2 \Delta u_1 \diff\m$, which is relevant to the proof of the lower bound of \eqref{eq:dAB}.
To overcome especially the latter point, we employ a linearisation technique for the heat semigroup in the Finsler case.
This approach is, however,  not applicable to the non-smooth setting in Theorem~\ref{t:m2} as we do not know if the linearised semigroup exists in metric measure spaces. 

\paragraph{The structure of the paper}%%%%%
After reviewing the basics of Finsler geometry in Section~\ref{sec:P}, 
we study the behaviour of the function $\bar\d_\m(A,B)$ in Section~\ref{sc:max}. 
Then, we prove the upper bound estimate in Theorem~\ref{t:m1} as well as Theorem~\ref{t:m2} in Section~\ref{sec:U}, 
and Section~\ref{sec:L} is devoted to the proof of the lower bound estimate in Theorem~\ref{t:m1}. 

\section*{Acknowledgement}
We are grateful to Karl-Theodor Sturm for his suggestion regarding Theorem~\ref{t:m2}, which helped us improve the assumption.
We also thank Davide Barilari for pointing out a couple of references concerning sub-Riemannian manifolds.
Lastly, we thank the anonymous referees for carefully reading the manuscript and providing us with constructive suggestions.

%%%%%%%%%%
\section{Preliminaries}\label{sec:P}%%%%%
%%%%%%%%%%

We first review the basics of Finsler geometry, and then introduce truncation functions as in \cite{HinRam03}
playing an essential role in the lower estimate in Section~\ref{sec:L}.

\subsection{Finsler manifolds}\label{ssc:Finsler}%%%%%
%%%%%
 
We refer the readers to~\cite{Oht21} for a concise description of the following contents, 
and also to \cite{GS01,OS09,OS14} for the behaviour of heat flow on Finsler manifolds. 

Let $\M$ be a connected $C^{\infty}$-manifold without boundary of dimension $n \ge 2$. 
Given local coordinates~$(x^i)_{i=1}^n$ on an open set~$U\subset \M$,
we will denote by $(x^i, v^j)_{i, j=1}^n$ the fibre-wise linear coordinates of the tangent bundle $TU$ given by
\[ v = \sum_{j=1}^n v^j \frac{\partial}{\partial x^j}\Bigl|_x \in T_x \M \comma \quad x \in U \fstop \]
We say that a nonnegative function~$F: T\M \to [0,\infty)$ is a \emph{$C^\infty$-Finsler structure} on $\M$ 
if the following three conditions hold:
\begin{enumerate}[(1)]
\item ({\it Regularity}) $F$ is $C^\infty$ on $T\M \setminus \{0\}$;
\item ({\it Positive $1$-homogeneity}) $F(cv)=cF(v)$ for every $v \in T\M$ and $c \ge 0$;
\item ({\it Strong convexity}) For every $v \in T\M \setminus \{0\}$,
the following $n \times n$ matrix is positive-definite:
\begin{equation}\label{eq:g_ij}
\bigl( g_{ij}(v)\bigr)^{n}_{i, j=1} 
 := \biggl( \frac{1}{2} \frac{\partial^2[F^2]}{\partial v^i\partial v^j}(v) \biggr)_{i,j=1}^n \fstop
\end{equation}
\end{enumerate}
We call a pair $(\M, F)$ a \emph{$C^\infty$-Finsler manifold}. 
We stress that the $1$-homogeneity is imposed only in the positive direction,
thereby $F(-v) \neq F(v)$ is allowed.
If $F(v)=F(-v)$ for all $v \in T\M$, then we say that $(\M, F)$ is \emph{reversible}. 
The matrix $(g_{ij}(v))$ in \eqref{eq:g_ij} provides an inner product $g_v$ of $T_x\M$ by 
\[ g_v\Biggl( \sum_{i=1}^n a_i \frac{\partial}{\partial x^i},\sum_{j=1}^n b_j \frac{\partial}{\partial x^j} \Biggr)
 := \sum_{i,j=1}^n a_i b_j g_{ij}(v) \fstop \]
We will also make use of their counterparts in the dual space $T^*_x \M$:
\begin{align}
&F^*(\alpha) :=\sup_{v \in T_x\M,\, F(v)=1} \alpha(v) \quad\ \text{for}\,\ \alpha \in T^*_x\M \comma
 \nonumber\\
&g^*_{ij}(\alpha) := \frac{1}{2} \frac{\partial^2 [(F^*)^2]}{\partial \alpha_i \partial \alpha_j}(\alpha)
 \quad\ \text{for}\,\ \alpha =\sum_{i=1}^n \alpha_i \diff x^i \in T^*_x\M \setminus \{0\} \comma
 \nonumber\\
&g^*_{\alpha} \Biggl( \sum_{i=1}^n a_i \diff x^i,\sum_{j=1}^n b_j \diff x^j \Biggr)
 := \sum_{i,j=1}^n a_i b_j g^*_{ij}(\alpha) \quad\ \text{on}\,\ T^*_x\M \fstop
 \label{eq:dual}
\end{align}
We remark that, though $\alpha(v) \le F^*(\alpha) F(v)$ holds by definition, 
$\alpha(v) \ge -F^*(\alpha) F(v)$ does not hold in general due to the irreversibility of $F$ 
(for example, one cannot replace the LHS of \eqref{e:KW3} with its absolute value). 

For $x,y \in \M$, we define
\[ \mssd(x, y) := \inf_{\eta}\int_0^1F \bigl( \dot{\eta}(t) \bigr) \diff t \comma \]
where the infimum is taken over all piecewise $C^1$-curves $\eta:[0,1] \to \M$ with $\eta(0)=x$ and $\eta(1)=y$.
Then $\d$ provides an \emph{asymmetric} distance function on $\M$,
namely the \emph{triangle inequality}
\[ \d(x,z) \le \d(x,y) +\d(y,z) \quad\ \text{for all}\,\ x,y,z \in \M \]
holds but $\d(y,x)$ may be different from $\d(x,y)$.
Note that $\d$ is symmetric if and only if $F$ is reversible.
We define the \emph{reversibility constant} of $(\M,F)$ as
\begin{equation}\label{eq:rev}
\Lambda_F := \sup_{v \in T\M \setminus \{0\}} \frac{F(v)}{F(-v)} 
 =\sup_{x,y \in \M,\, x \neq y} \frac{\d(x,y)}{\d(y,x)} \fstop
\end{equation}
Observe that $\Lambda_F \in [1,\infty]$ in general, and $\Lambda_F =1$ holds only in the reversible case. 

We say that $(\M,F)$ is \emph{forward complete} if any closed, forward bounded set $A \subset \M$ 
(i.e., $\sup_{y \in A}\d(x,y) <\infty$ for some or, equivalently, any $x \in \M$) is compact. 
The \emph{backward completeness} is defined as the forward completeness of $(\M,\overline{F})$, 
where $\overline{F}$ is the \emph{reverse} Finsler structure given by $\overline{F}(v):=F(-v)$. 
If $\Lambda_F <\infty$, then the forward and backward completenesses are mutually equivalent 
and we may simply call it the \emph{completeness}. 

\paragraph{Uniform convexity and smoothness}%%%%%
%%%%%
We define the \emph{uniform convexity} and \emph{smoothness constants} of $(\M,F)$ as 
\begin{equation}\label{eq:CS}
\mathsf{C}_F :=\sup_{x \in \M} \sup_{v,w \in T_x\M \setminus \{0\}} \frac{F(w)^2}{g_v(w,w)} \comma \qquad
\mathsf{S}_F :=\sup_{x \in \M} \sup_{v,w \in T_x\M \setminus \{0\}} \frac{g_v(w,w)}{F(w)^2} \comma
\end{equation}
respectively. 
Since $g_v$ comes from the Hessian of $F^2$,
$\mathsf{C}_F$ (resp.\ $\mathsf{S}_F$) actually measures the convexity (resp.\ concavity) of $F^2$ in tangent spaces. 
We have $\mathsf{C}_F, \mathsf{S}_F \in [1,\infty]$ in general, 
and $\mathsf{C}_F =1$ or $\mathsf{S}_F =1$ holds only in Riemannian manifolds 
(see \cite[Proposition 1.6]{Oht21}). 
On a compact Finsler manifold, $\mathsf{C}_F$ and $\mathsf{S}_F$ are finite
thanks to the smoothness and the strong convexity of $F$.
We also remark that their dual expressions are given by 
\begin{equation}\label{eq:CSdual}
\mathsf{C}_F =\sup_{x \in \M} \sup_{\alpha,\beta \in T^*_x\M \setminus \{0\}}
 \frac{g^*_{\alpha}(\beta,\beta)}{F^*(\beta)^2} \comma \qquad
\mathsf{S}_F =\sup_{x \in \M} \sup_{\alpha,\beta \in T^*_x\M \setminus \{0\}}
 \frac{F^*(\beta)^2}{g^*_{\alpha}(\beta,\beta)} \fstop
\end{equation}
We refer to \cite{Oht09}, \cite[\S 8.3.2]{Oht21} for more discussions on $\mathsf{C}_F$ and $\mathsf{S}_F$.
The reversibility constant $\Lambda_F$ in \eqref{eq:rev}
can be bounded by $\mathsf{C}_F$ and $\mathsf{S}_F$ as (see \cite[Lemma~8.18]{Oht21}) 
\[ \Lambda_F \le \min\bigl\{ \sqrt{\mathsf{C}_F} \comma \sqrt{\mathsf{S}_F} \bigr\} \fstop \]

\paragraph{Gradient vectors}%%%%%
%%%%%
For a differentiable function~$f: \M \to \R$, its \emph{gradient vector} at $x \in \M$ is defined
to be the Legendre transform of the derivative of $f$:
\[ \nabla f(x):=\mathcal L^* \bigl( \dif f(x) \bigr) \in T_x\M \fstop \]
Here the \emph{Legendre transform} $\mathcal L^*: T^*_x \M \to T_x \M$ maps 
$\alpha \in T_x^* \M$ to the unique element $v \in T_x \M$ 
such that $F(v) = F^*(\alpha)$ and $\alpha(v)=F^*(\alpha)^2$. 
We remark that $(g_{ij}(\mathcal{L}^*(\alpha)))$ is the inverse matrix of $(g^*_{ij}(\alpha))$, 
provided $\alpha \neq 0$. 
In local coordinates, we can write down 
\[ \nabla f(x) =\sum_{i,j=1}^n g^*_{ij}\bigl( \dif f(x) \bigr) \frac{\partial f}{\partial x^j}(x)
 \frac{\partial}{\partial x^i} \Big|_x \]
(when $\dif f(x) \neq 0$; while $\nabla f(x)=0$ if $\dif f(x)=0$). 
We say that $f$ is \emph{$1$-Lipschitz} if 
\[ f(y)-f(x) \le \d(x,y) \quad\ \text{for all}\,\ x,y \in \M \comma \]
which is equivalent to $F(\nabla f) \le 1$ when $f$ is differentiable. 
For example, given $x_0 \in \M$, the functions 
$x \mapsto \d(x_0,x)$ and $x \mapsto -\d(x,x_0)$ are $1$-Lipschitz by the triangle inequality. 

\begin{rem}[Non-linearity of~$\nabla$]\label{rm:nabla}
As the differential $\dif$ stems only from the differentiable structure of $\M$,  
it is a linear operator and enjoys the chain and Leibniz rules. 
However, the Legendre transform $\mathcal{L}^*$ is non-linear
($\mathcal{L}^*(\alpha+\beta) \neq \mathcal{L}^*(\alpha) +\mathcal{L}^*(\beta)$)
and irreversible ($\mathcal{L}^*(-\alpha) \neq \mathcal{L}^*(\alpha)$).
Therefore, the gradient operator $\nabla$ does not satisfy the Leibniz rule, 
and the chain rule holds only for non-decreasing functions:
\[ \nabla (\phi \circ f) =\mathcal{L}^* (\phi' \circ f \cdot \dif f) =\phi' \circ f \cdot \nabla f
 \quad\ \text{if}\ \phi' \ge 0 \comma \] 
while we have $\nabla (\phi \circ f) =-\phi' \circ f \cdot \nabla(-f)$ if $\phi' \le 0$. 
In the reversible case, the chain rule holds regardless of the sign of $\phi'$. 
\end{rem}

\paragraph{Heat semigroup}%%%%%
%%%%%
Now, we fix a positive $C^{\infty}$-measure $\m$ on $\M$ 
(in the sense that, in each local chart, 
$\m=\rho \diff x^1 \cdots \diff x^n$ for a positive $C^{\infty}$-function $\rho$). 
Then the \emph{divergence} of a differentiable vector field $V$ on $\M$
with respect to $\m$ is defined in local coordinates as 
\[ \div_{\mssm}V := \sum_{i=1}^n \biggl( \frac{\partial V^i}{\partial x^i} + V^i\frac{\partial \psi}{\partial x^i} \biggr) \comma \]
where $V = \sum_{i=1}^n V^i (\partial/\partial x^i)$ 
and we wrote $\m=e^{\psi} \diff x^1 \cdots \diff x^n$ in the coordinates. 
It can be generalised to measurable vector fields $V$ in the distributional sense 
(against $C^{\infty}$-functions of compact support) as 
\[ \int_\M \phi\, \div_{\mssm} V \diff \mssm = - \int_\M \dif\phi(V) \diff \mssm 
\quad\ \text{for all}\,\ \phi \in C_c^\infty(\M) \fstop \]
Then we define the \emph{distributional Laplacian} $\Delta:=\div_\m \circ \nabla$, i.e., 
\[ \int_\M \phi \Delta u \diff \mssm := - \int_\M \dif \phi(\nabla u) \diff \mssm 
\quad\ \text{for all}\,\ \phi \in C_c^\infty(\M) \fstop \]
This Laplacian is again non-linear by the non-linearity of $\nabla$. 
Moreover, 
\begin{equation}\label{eq:asym}
\int_\M u_1 \Delta u_2 \diff\m \neq \int_\M u_2 \Delta u_1 \diff\m 
\end{equation}
in general (in other words, $\dif u_1(\nabla u_2) \neq \dif u_2(\nabla u_1)$). 

Define the \emph{energy functional} $\E$ on $H^1_\loc(\M)$ as 
\[ \E(u) := \frac{1}{2} \int_\M F(\nabla u)^2 \diff \mssm = \frac{1}{2} \int_\M F^*(\dif u)^2 \diff \mssm \]
($H^1_\loc(\M)$ is defined solely by the differentiable structure of $\M$ via local charts). 
The \emph{Sobolev space} $H^1(\M)$ is defined as the set of functions $u \in L^2(\M) \cap H^1_{\loc}(\M)$ 
such that $\E(u)+\E(-u) <\infty$. 
Denote by $H^1_0(\M)$ the closure of $C_c^\infty(\M)$ with respect to the norm 
\[ \|u\|_{H^1} := \sqrt{ \|u\|_{L^2}^2 +\E(u) +\E(-u)} \fstop \]
Note that $(H^1(\M),\|\cdot\|_{H^1})$ is not a Hilbert space but a reflexive Banach space. 
If $\Lambda_F <\infty$ and $(\M,F)$ is complete, then we have $H^1_0(\M)=H^1(\M)$ 
(see \cite[Lemma~11.4]{Oht21}). 

The \emph{$L^2$-heat semigroup} $u_t=\T_t f$ is defined as the solution to the Cauchy problem: 
\[ \partial_t u_t = \frac{1}{2} \Delta u_t \comma \qquad
 u_0 = f \fstop \]
One can construct $(u_t)_{t \ge 0}$ as gradient flow of the energy $\E$ in $L^2(\M)$, 
provided $\Lambda_F<\infty$. 
Precisely, we first construct $\T_t f$ for $f \in H^1_0(\M)$ 
and then extend it to a contraction semigroup acting on $L^2(\M)$ 
(see, e.g., \cite{AmbGigSav08}, \cite[\S 13.2]{Oht21}). 
Thanks to the regularizing effect as in \cite[Theorem~4.0.4]{AmbGigSav08}, \cite[(4.26)]{AmbGigSav14}, 
we have $\T_t L^2(\M) \subset H^1_0(\M)$ for $t>0$. 
We stress that the heat semigroup~$(\T_t)_{t \ge 0}$ is non-linear. 
Indeed, $\T_t(f_1 +f_2) = \T_t f_1 +\T_t f_2$ does not hold, 
and $\T_t(cf) =c\T_t f$ holds only when $c \ge 0$ due to the irreversibility. 
Moreover, $\T_t f$ is not smooth at points where $\dif [\T_t f]$ vanishes.

\begin{rem}\label{rm:GF}
Besides the above interpretation of heat flow as gradient flow of the energy, 
one can also regard heat flow as gradient flow of the relative entropy in the $L^2$-Wasserstein space. 
Precisely, in the Finsler case, we need to consider the Wasserstein space 
for the \emph{reverse} Finsler structure $\overline{F}(v)=F(-v)$ (we refer to \cite{OS09,OZ22} for details). 
This fact could be compared with the appearance of $\d(A,B)$, rather than $\d(B,A)$, in our results.
\end{rem}

\paragraph{Linearised heat semigroups}%%%%%
%%%%%
In the last step of the lower estimate in Section~\ref{sec:L}, 
we will employ a linearisation of the heat semigroup to overcome a difficulty due to the asymmetry~\eqref{eq:asym}. 
See \eqref{eq:linear} in Lemma \ref{l:215} below for the precise equation, 
here we recall a related result from \cite[\S 13.5]{Oht21}. 

Let $(u_t)_{t \ge 0}$ be a solution to the heat equation, 
and take a measurable one-parameter family $(V_t)_{t \ge 0}$ of nowhere vanishing vector fields 
such that $V_t(x) =\nabla u_t(x)$ for all $x$ with $\dif u_t(x) \neq 0$. 
Given $f \in H^1_0(\M)$, there exists a (weak) solution $(f_t)_{t \ge 0} \subset H^1_0(\M)$ 
to the \emph{linearised heat equation} 
\begin{equation}\label{eq:lin-heat}
\partial_t f_t =\frac{1}{2} \Delta\!^{\nabla u_t} f_t \comma \qquad f_0=f \comma 
\end{equation}
where $\Delta\!^{\nabla u_t}:=\div_\m \circ \nabla^{\nabla u_t}$ is the \emph{linearised Laplacian} defined with 
\begin{equation}\label{eq:lin-gra}
\nabla^{\nabla u_t} h := \sum_{i,j=1}^n g^*_{ij} \bigl( (\mathcal{L}^{*})^{-1}(V_t) \bigr)
 \frac{\partial h}{\partial x^j} \frac{\partial}{\partial x^i} \fstop 
\end{equation}
Note that we suppressed the dependence on the choice of $V_t$ for simplicity. 
We also remark that $\Delta\!^{\nabla u_t} u_t =\Delta u_t$.
Linearised heat semigroups play an essential role in geometric analysis on Finsler manifolds, 
including gradient estimates \cite{OS14} as well as functional and geometric inequalities \cite{Oht17a,Oht22}. 

\subsection{Truncation functions}\label{ssc:cut}%%%%%
%%%%%

We shall introduce some useful functions as in \cite[\S 2.1]{HinRam03}. 
Let $\zeta: [0,\infty) \to [0,\infty)$ be a bounded concave $C^3$-function satisfying: 
\begin{enumerate}[{\rm (1)}]
\item $\zeta(t)=t$ for $t \in [0,1]$ and $0<\zeta'(t)\le 1$ for all $t \ge 0$;
\item There is a positive constant $C$ such that $0 \le -\zeta''(t) \le C\zeta'(t)$ for all $t \ge 0$.
\end{enumerate}
For instance, any $C^{\infty}$-function $\zeta$ such that 
$\zeta(0)=0$ and $\zeta'$ is non-increasing with 
\begin{align*}
\zeta'(t)=
\begin{cases}
1 \quad & 0 \le t \le 1 \comma \\
e^{-t} \quad & t \ge 2 
\end{cases}
\end{align*}
satisfies the required properties. 
By these conditions, the monotone limit $\lim_{t \to \infty} \zeta(t)=L$ exists. 
Define $\varphi^K(t) :=K\zeta(t/K)$ for $K>0$. 
For notational simplicity, we do not write $K$ explicitly when no confusion could occur. 
We also set
\[ \Phi(t) :=\int_0^t \varphi'(s)^2 \diff s \comma \qquad \Psi(t) :=t\varphi'(t)^2 \fstop \]
Then we immediately have the following estimates:
\begin{align} \label{e:estimates}
&0<\varphi'(t) \le 1\comma \quad 0 \le -\varphi''(t) \le \frac{C}{K}\varphi'(t) \comma \quad 
 \Phi(t)=\Psi(t)=t \,\ \text{on}\ [0, K] \comma 
\\ 
&0 \le \Psi(t) \le \Phi(t) \le \int_0^t \varphi'(s)\diff s = \varphi(t) \le KL \fstop \notag
\end{align}

%%%%%%%%%%
\section{Maximal functions}\label{sc:max}%%%%%
%%%%%%%%%%

In this section, we assume $\Lambda_F <\infty$ and the completeness of $(\M,F)$. 
To begin the proof of Theorem~\ref{t:m1}, 
fix $x_0 \in \M$ and let $(B_{r_k}(x_0))_{k \in \N}$ and $(\chi_k)_{k \in \N}$ be sequences 
of open forward balls (i.e., $B_r(x):=\{ y \in \M : \d(x,y)<r \}$)  and functions such that $\lim_{k \to \infty} r_k =\infty$, 
$0 \le \chi_k \le 1$, $\chi_k \equiv 1$ on $B_{r_k}(x_0)$, $\chi_k \equiv 0$ on $\M \setminus B_{r_{k+1}}(x_0)$, 
and $-\chi_k$ is $1$-Lipschitz. 
Note that, in particular, $\chi_k$ is compactly supported. 
Throughout this article, for simplicity, 
we omit an arbitrarily fixed centre $x_0$ and write $B_k:=B_{r_k}(x_0)$. 

We next introduce a distance-like function $\bar\d_B$ for a measurable set $B \subset \M$. 
We will always assume $0<\m(A),\m(B)<\infty$. 
For $R \ge 0$, define 
\begin{align}\label{eq:L_AM}
\L_{B,R} :=\{f \in \L: f=0 \ \text{a.e.~on $B$ and $0 \le f \le R$ a.e.} \} \fstop
\end{align}
Recall the definitions \eqref{eq:d_m}, \eqref{eq:L} of $\bar\d_\m(A, B)$ and $\L$ given in the introduction. 
We also set 
\[
\d_\m(A, B) :=\essinf_{x \in A} \inf_{y \in B} \d(x,y) \fstop 
\]

\begin{prop} \label{pr:dAB}
%Suppose that~$W^{1,2}(\M)$ is reflexive. 
For any measurable set $B \subset \M$, 
there exists a unique $[0,\infty]$-valued measurable function~$\bar\d_B$ such that, 
for every $R>0$, the function $\bar\d_B \wedge R$ is the maximal element of $\L_{B,R}$. 
Precisely, $\bar\d_B \wedge R \in \L_{B,R}$ and $f \le \bar\d_B \wedge R$ holds a.e.\ for any $f \in \L_{B,R}$. 
Furthermore, for any measurable set $A \subset \M$, we have 
\[ \bar\d_\m(A,B) =\essinf_{x \in A} \bar\d_B(x) \fstop \]
\end{prop}

\begin{proof}
We can follow the lines of \cite[Proposition~3.11]{AriHin05}
by replacing $E_k$, $\mathbb D$, $\mathbb D_{A, M}$ and $\mathbb D_0$ there
by $B_k$, $H^1_0(\M)=H^1(\M)$, $\mathbb L_{B,R}$ and $\L$. 
\end{proof}

\begin{rem}\label{rm:to/from}
The requirement $F^*(-\dif f) \le 1$ in \eqref{eq:L} means that 
$\bar\d_B$ is regarded as the distance ``to'' the set $B$ rather than the distance ``from'' $B$.
We need to distinguish them in the present situation where the distance function $\d$ is asymmetric. 
\end{rem}

Set $\d_B(x):=\inf_{y \in B}\d(x, y)$ for $x \in \M$. 
Note that $-\d_B$ is $1$-Lipschitz by the triangle inequality, 
and $F^*(-\dif \d_B) \le 1$ a.e. 
We compare $\d_B$ with $\bar\d_B$ obtained in~Proposition~\ref{pr:dAB}. 

\begin{lem} \label{lm:DM}
%Suppose that~$W^{1,2}(\M)$ is reflexive. 
For any measurable set $B \subset \M$, we have $\d_B \le \bar\d_B$ a.e. 
In particular, for any measurable sets $A,B \subset \M$, we have 
\begin{align*} %\label{e:MD0}
\d_\m(A, B) \le \bar\d_\m(A, B) \fstop
\end{align*}
\end{lem}

\begin{proof}
Since $\d_B \wedge R \in \L_{B,R}$ for every $R>0$, 
we have $\d_B \wedge R \le \bar\d_B \wedge R$  by the maximality of $\bar\d_B \wedge R$ in $\L_{B,R}$. 
Letting $R \to \infty$, we obtain $\d_B \le \bar\d_B$ a.e. 
The latter assertion then follows from the definition of $\d_\m(A,B)$ and Proposition~\ref{pr:dAB} as
\[ \d_\m(A, B) = \essinf_{x \in A} \d_B(x) \le \essinf_{x \in A} \bar\d_B(x) = \bar\d_\m(A, B) \fstop \]
\end{proof}

For open sets, these distance-like functions actually coincide with the usual distance 
$\d(A,B)=\inf_{x \in A,\, y \in B}\d(x, y) =\inf_{x \in A} \d_B(x)$. 

\begin{lem} \label{lm:CDD}
%Suppose that~$W^{1,2}(\M)$ is reflexive and \eqref{sl} holds. 
For any open set $B \subset \M$, we have $\bar\d_B = \d_B$ a.e. 
In particular, for any open sets $A,B \subset \M$, we have 
\begin{align*} %\label{e:MD1}
\bar\d_\m(A, B) = \d_\m(A, B) = \d(A,B) \fstop
\end{align*}
\end{lem}

\begin{proof}
Since $\bar\d_B \wedge R \in \L_{B,R}$ and $B$ is open, from the proposition below, 
its continuous version $f$ satisfies $f \equiv 0$ on $B$, and $-f$ is $1$-Lipschitz. 
Hence, for all $y \in B$ and a.e.\ $x \in \M$, we have 
\begin{align*}
\bar\d_B(x) \wedge R =  f(x) = f(x)-f(y) \le \d(x, y) \fstop
\end{align*}
Letting $R \to \infty$ and taking the infimum in $y \in B$, we conclude $\bar\d_B \le \d_B$ a.e. 
Combining this with Lemma~\ref{lm:DM}, we deduce the former assertion. 
Then $\bar\d_\m(A, B) = \d_\m(A, B)$ follows from the definition of $\d_\m(A, B)$ and Proposition~\ref{pr:dAB}, 
while $\d_\m(A,B)=\d(A,B)$ is immediate since $A$ is open and $\d_B$ is continuous. 
\end{proof}

Though the next proposition (called the \emph{local Sobolev-to-Lipschitz property}) 
should be a known fact, we give a short proof for completeness. 

\begin{prop}\label{pr:StL}
For any $f \in \L$, there exists a bounded $1$-Lipschitz function $-\tilde{f}$ such that $\tilde{f}=f$ a.e.
\end{prop}

\begin{proof}
By multiplying with a smooth cut-off function in each local chart, 
one can reduce the existence of a (locally) Lipschitz function $\tilde{f}$ such that $\tilde{f}=f$ a.e.\ 
to the Euclidean case. 
The Euclidean case can be seen, e.g., in \cite[Theorem~4.5]{EvaGar15}. 
Then $F^*(-\dif \tilde{f}) =F^*(-\dif f) \le 1$ a.e.\ yields that $-\tilde{f}$ is $1$-Lipschitz. 
\end{proof}

%%%%%%%%%%
\section{Upper estimate} \label{sec:U}%%%%%
%%%%%%%%%%

For measurable sets $A,B \subset \M$, recall that we define 
\[ \P_t(A, B):=\int_A \T_t \mathbf 1_B \diff \mssm \fstop\]
In this section, we establish the upper estimate of~\eqref{eq:dAB}. 
Our proof is essentially along the lines of \cite[Theorem~4.1]{AriHin05} or \cite[Theorem~2.8]{HinRam03}, 
where the former is a generalisation of the latter to admit infinite total mass. 
The upper estimate is less demanding than the lower estimate. 
In fact, after establishing the upper estimate in the Finsler setting, 
we will see that it also holds true in metric measure spaces 
under mild assumptions in Subsection~\ref{subsec:MM}. 

%%%%%
\subsection{Finsler case}\label{ssc:U_Fin}%%%%%
%%%%%

Let $(\M,F,\m)$ be a complete measured Finsler manifold such that $\Lambda_F <\infty$ 
equipped with a measure $\m$. 

\begin{prop}\label{pr:U}
For any measurable sets $A, B \subset \M$ with $0<\m(A),\m(B)<\infty$ and $t>0$, we have 
\[ \P_t(A, B) \le \sqrt{\mssm(A)}\sqrt{\mssm(B)} \exp\biggl( -\frac{\bar\d_\m(A,B)^2}{2t} \biggr) \fstop \]
In particular,
\[ \limsup_{t \downarrow 0}t\log \P_t(A, B) \le -\frac{\bar\d_\m(A,B)^2}{2} \fstop \]
\end{prop}

\begin{proof}
Given $R>0$, take $f \in \L_{B, R}$ and set $f_k:=f \chi_k \in H^1_0(\M)$ for $k \in \N$, 
with $\chi_k$ chosen in Section~\ref{sc:max}. 
We also set $u_t:=\T_t\1_B \in H^1_0(\M)$ and $u_{t, k} :=u_t \chi_k\in H^1_0(\M)$. 
Note that $0 \le u_t \le 1$ a.e.\ (see, e.g., \cite[Lemma 13.13]{Oht21}). 
For $\alpha \ge 0$ (chosen later), we consider the function
\[ \xi(t) := \int_\M (e^{\alpha f}u_t)^2 \diff \mssm \fstop \]
By the heat equation and the integration by parts, we have 
\begin{align} \label{e:KW0}
\xi'(t)
&= \int_\M e^{2\alpha f}u_t \Delta u_t \diff \mssm
 = \lim_{k \to \infty} \int_\M e^{2\alpha f_k}u_{t, k} \Delta u_t \diff \mssm
\\
&=-\lim_{k \to \infty}\int_\M \dif (e^{2\alpha f_k}u_{t, k})(\nabla u_t)  \diff \mssm \notag \fstop
\end{align}
Let us now see that 
\begin{align}\label{e:KW}
\lim_{k \to \infty}\int_\M \dif (e^{2\alpha f_k}u_{t, k})(\nabla u_t)  \diff \mssm
 = \lim_{k \to \infty}\int_\M \dif (e^{2\alpha f_k}u_{t, k})(\nabla u_{t, k})  \diff \mssm  \fstop
\end{align}
Note first that
\begin{align*}%\label{e:KW1}
&\biggl| \int_\M \dif (e^{2\alpha f_k}u_{t, k})(\nabla u_t)  \diff \mssm
 -\int_\M \dif (e^{2\alpha f_k}u_{t, k})(\nabla u_{t, k})  \diff \mssm \biggr| \\
&\le \Lambda_F \int_\M F^* \bigl( \dif (e^{2\alpha f_k}u_{t, k}) \bigr) F(\nabla u_t- \nabla u_{t, k}) \diff \mssm \\
&\le \Lambda_F \sqrt{2\E (e^{2\alpha f_k}u_{t, k})} \cdot \| F(\nabla u_t- \nabla u_{t, k}) \|_{L^2} \fstop
\end{align*}
We find from $u_t \in H^1_0(\M)$, 
\[ F^*(-\dif f_k) \le F^*(-f \dif \chi_k) +F^*(-\chi_k \dif f) \le R+1 \quad \text{a.e.} \]
and $\Lambda_F<\infty$ that $\E(e^{2\alpha f_k}u_{t, k})$ is bounded above uniformly in $k$. 
Moreover, since $\dif u_t = \dif u_{t, k}$ on $B_k$, we observe 
\begin{align*}
\| F(\nabla u_t- \nabla u_{t, k}) \|_{L^2}^2 
&\le \int_{\M \setminus B_k} \bigl( F(\nabla u_t) +\Lambda_F F(\nabla u_{t,k}) \bigr)^2 \diff \m \\
&\le \int_{\M \setminus B_k} \Bigl( F(\nabla u_t)
 +\Lambda_F \bigl( u_t F^*(\dif \chi_k) +\chi_k F^*(\dif u_t) \bigr) \Bigr)^2 \diff \m \\
&\le \int_{\M \setminus B_k} \bigl( \Lambda_F^2 u_t +(\Lambda_F +1)F(\nabla u_t) \bigr)^2 \diff \m
 \,\xrightarrow{k\to\infty}\, 0 \fstop 
\end{align*}
Therefore, \eqref{e:KW} has been shown and the RHS of \eqref{e:KW0} can be expanded as 
\begin{align}
&-\lim_{k \to \infty}\int_\M \dif (e^{2\alpha f_k}u_{t, k})(\nabla u_{t, k})  \diff \mssm \label{e:KW2}\\
&= -\lim_{k \to \infty}\int_\M e^{2\alpha f_k} \bigl( \dif u_{t, k}(\nabla u_{t, k})  +2\alpha u_{t, k} \,\dif f_k(\nabla u_{t, k})\bigr)\diff \mssm \fstop \nonumber
\end{align}
On $B_k$, it follows from $\chi_k \equiv 1$ and $f \in \L_{B,R}$ that $F^*(-\dif f_k) =F^*(-\dif f) \le 1$, 
thereby, 
\begin{align} \label{e:KW3}
-\dif f_k(\nabla u_{t,k}) \le F^*(-\dif f_k) F(\nabla u_{t,k}) \le F(\nabla u_{t,k}) \fstop
\end{align}
Combining \eqref{e:KW0}--\eqref{e:KW3}, we obtain 
\begin{align*}
\xi'(t) 
&= -\lim_{k \to \infty}\int_{\M} e^{2\alpha f_k} \bigl( \dif u_{t, k}(\nabla u_{t, k})
 +2\alpha u_{t, k} \,\dif f_k(\nabla u_{t, k})\bigr) \diff\mssm \\
&\le -\lim_{k \to \infty}\int_{\M} e^{2\alpha f_k} 
 \bigl( F(\nabla u_{t,k})^2  -2\alpha u_{t,k} F(\nabla u_{t,k}) \bigr) \diff\mssm 
\\
&\le \lim_{k \to \infty}\int_\M e^{2\alpha f_k} (\alpha u_{t, k})^2 \diff\mssm
\\
&= \alpha^2 \int_\M e^{2\alpha f} u_t^2\diff \mssm
= \alpha^2 \xi(t) \comma
\end{align*}
which yields 
\[ \xi(t) \le \xi(0)e^{\alpha^2 t} \fstop \]

Now, we take $f= \bar\d_B\wedge \bar\d_\m(A,B) \in \L_{B,\bar\d_\m(A,B)}$. 
Since $\bar\d_B =0$ a.e.\ on $B$ and $\bar\d_B \ge \bar\d_\m(A,B)$ a.e.\ on $A$ by Proposition~\ref{pr:dAB}, 
we have 
\begin{align*}
\P_t(A, B) &= \int_\M u_t \1_A \diff \mssm \le \|e^{\alpha f} u_t\|_{L^2} \|e^{-\alpha f}\1_A\|_{L^2} 
\le \sqrt{\xi(0)} e^{\alpha^2 t/2} \sqrt{\mssm(A)} e^{-\alpha \bar\d_\m(A,B)} 
\\
&= \sqrt{\mssm(A)}\sqrt{\mssm(B)} \exp \biggl( \frac{\alpha^2 t}{2} -\alpha \bar\d_\m(A,B) \biggr) \fstop
\end{align*}
Choosing the optimal value $\alpha =\bar\d_\m(A,B)/t$, we conclude 
\begin{align*}%\label{ineq:UBM}
\P_t(A, B) \le \sqrt{\mssm(A)}\sqrt{\mssm(B)} \exp \biggl( -\frac{\bar\d_\m(A, B)^2}{2t} \biggr) \fstop 
\end{align*}
\end{proof}

Set $\Phi_t :=\Phi(-t\log \T_t \1_B)$ for $t>0$ and $\Phi$ as in Subsection~\ref{ssc:cut}. 
Since $\Phi$ is bounded, for any finite measure $\nu$ mutually absolutely continuous with $\m$, 
$(\Phi_t)_{t>0}$ is uniformly bounded and hence weakly relatively compact in $L^2(\nu)$. 
The following corollary will play a role in the lower estimate.

\begin{cor}\label{c:1}
For any measurable set $B \subset \M$ with $0<\m(B)<\infty$ 
and any weak $L^2(\nu)$-limit $\Phi_0$ of $(\Phi_t)_{t>0}$ as $t \to 0$, we have 
\[ \Phi_0 \ge \Phi \biggl( \frac{\bar\mssd_B^2}{2} \biggr) \quad \text{a.e.} \]
\end{cor}

\begin{proof}
We can follow the lines of \cite[Lemma~2.9]{HinRam03}
thanks to Propositions~\ref{pr:dAB}, \ref{pr:U}.
\end{proof}

%%%%%
\subsection{Nonsmooth case}\label{subsec:MM}%%%%%
%%%%%

In this subsection, we briefly explain that the upper estimate (Proposition \ref{pr:U}) 
can be generalised to the nonsmooth setting of metric measure spaces 
by utilising the differential calculus developed in \cite{AmbGigSav14, Gig18}. 

\paragraph{Setting}%%%%%
Let $(\X, \d)$ be a complete separable metric space (with a usual symmetric distance function), 
and $\m$ be a fully supported Borel measure on $\X$ such that $\m(B)<\infty$ 
for every bounded $\m$-measurable set $B \subset \X$. 
A Borel probability measure $\pi$ on the set $C([0,1], \X)$ 
of continuous curves $\gamma:[0,1] \to \X$ is called a \emph{test plan} 
if there is a constant $C>0$ such that $(e_t)_{\#} \pi \le C \m$ for all $t \in [0,1]$ and 
\[ 
 \int_{C([0,1], \X)} \int_0^1 |\dot{\gamma}_t|^2 \diff t \,\pi(\diff\gamma) <\infty \comma 
\]
where $e_t(\gamma):=\gamma_t$ is the evaluation map, 
$(e_t)_{\#} \pi$ denotes the push-forward of $\pi$ by $e_t$, 
and $|\dot{\gamma}_t|:=\lim_{s \to t}\d(\gamma_s,\gamma_t)/|s-t|$ is the metric speed. 

The \emph{Sobolev class} $S^{2}(\X)$ consists of measurable functions $f$ 
such that there is a nonnegative function $g \in L^2(\X)$ satisfying
\[ 
\int_{C([0,1], \X)} |f(\gamma_1)-f(\gamma_0)| \,\pi(\diff\gamma)
 \le \int_{C([0,1], \X)} \int_{0}^1 g(\gamma_t)|\dot{\gamma}_t| \diff t \,\pi(\diff\gamma) 
\]
for all test plans $\pi$. 
The minimal function $g$ is called the \emph{minimal weak upper gradient} and denoted by $|\Dif f|$. 
We define the \emph{Sobolev space} $W^{1,2}(\X) :=S^{2}(\X) \cap L^2(\X)$ equipped with the norm
\[ \|f\|_{W^{1,2}} :=\sqrt{\|f\|_{L^2}^2 + \bigl\| |\Dif f| \bigr\|_{L^2}^2} \fstop \]
(Denoting the Sobolev space by $W^{1,2}(\X)$ follows the notation in \cite{AmbGigSav14,Gig18}, 
while in the Finsler setting we use $H^1(\M), H^1_0(\M)$ as in \cite{Oht21}.)
We remark that $W^{1,2}(\X)$ is not necessarily separable, reflexive, nor a Hilbert space in this generality 
(see \cite{ACDM}, \cite[Theorem~2.1.5]{Gig18}). 

\paragraph{Heat semigroup}%%%%%
The \emph{Cheeger energy} $\Ch: L^2(\X) \to [0, \infty]$ is defined as 
\[%\label{eq:Ch}
\Ch(f) := \frac{1}{2}\int_{\X} |\Dif f|^2 \diff \m 
\]
for $f \in W^{1,2}(\X)$, and $\Ch(f):=\infty$ otherwise. 
Note that $\Ch$ is convex, lower semi-continuous and the domain $W^{1,2}(\X)$ is dense in $L^2(\X)$. 
For $f \in W^{1,2}(\X)$ such that the sub-differential $\partial^- \Ch(f) \subset L^2(\X)$ is nonempty,
$\Delta f \in L^2(\X)$ is defined as $\Delta f := -h$, 
where $h$ is the element of minimal $L^2$-norm in $\partial^-\Ch(f)$. 
Note that this Laplacian is not necessarily linear. 

Thanks to the theory of gradient flows for convex functions on Hilbert spaces 
(we refer to \cite{AmbGigSav08}), 
we have the \emph{heat semigroup} $(\T_t)_{t \ge 0}$ of continuous operators from $L^2(\X)$ to itself 
such that, for every $f \in L^2(\X)$, $t \mapsto \T_t f \in L^2(\X)$ is continuous on $[0, \infty)$, 
absolutely continuous on $(0, \infty)$ and 
\[ 
\frac{\diff}{\diff t} \T_t f =\frac{1}{2} \Delta \T_t f \quad \text{a.e.\ $t>0$} \fstop 
\]

\paragraph{Differentials and gradients}%%%%%
According to \cite[Definition~2.2.1]{Gig18}, 
there exists a dual pair of Banach spaces called the \emph{tangent module} and the \emph{cotangent module}:
\[ 
\bigl( L^2(T\X), \|\cdot\|_{L^2(T\X)} \bigr) \comma \qquad
 \bigl( L^2(T^*\X), \|\cdot\|_{L^2(T^*\X)} \bigr) \fstop 
\]
These spaces are endowed with maps
%structures, 
%namely they are $L^\infty(\X)$-premodules and there exist maps 
$|\cdot|: L^2(T\X) \to L^2(\X)$, $|\cdot|_*: L^2(T^*\X) \to L^2(\X)$ such that 
\[
 \bigl\| |v| \bigr\|_{L^2}=\|v\|_{L^2(T\X)} \,\ \text{for}\ v \in L^2(T\X) \comma \qquad
 \bigl\| |\omega|_* \bigr\|_{L^2}=\|\omega\|_{L^2(T^*\X)} \,\ \text{for}\ \omega \in L^2(T^*\X) \fstop 
\]
There exist a linear operator $\dif: S^2(\X) \to L^2(T^*\X)$ 
and a (not necessarily linear) multi-valued operator ${\rm Grad}: S^2(\X) \to L^2(T\X)$ satisfying 
\[
 \dif f(v) = |v|^2 = |{\dif} f|_*^2 =|\Dif f|^2 \quad\ \text{for}\ v \in {\rm Grad} f \comma f \in S^2(\X) \fstop
\]
%for $$. 

We say that $(\X, \d,\m)$ is \emph{infinitesimally strictly convex} if 
${\rm Grad}(f)$ consists of exactly one element, which is denoted by $\nabla f$ 
(see, e.g., \cite[Definition~2.3.9]{Gig18}). 
In this case, we have 
\[
\Ch(f) = \frac{1}{2}\int_{\X} \dif f(\nabla f) \diff \m, \quad\ f \in W^{1,2}(\X) \fstop
\]
The \emph{differential} operator $\dif$ enjoys the locality as well as the Leibniz and chain rules 
in an appropriate sense. 
On the other hand, the \emph{gradient} operator $\nabla$ satisfies the locality and the chain rule, 
whereas the Leibniz rule does not hold. 

\paragraph{Upper estimate}%%%%%
Define $W^{1,2}_{\loc}(\X)$ as the set of functions $f$ admitting $(f_k)_{k \in \N} \subset W^{1,2}(\X)$ 
such that $f=f_k$ on $B_k$ (with $B_k$ as in Section~\ref{sc:max}; 
$W^{1,2}_{\loc}(\X)$ does not depend on the choice of $B_k$). 
Then we define $\bar\mssd_\mssm(A, B)$ as in~\eqref{eq:d_m} with
\[
 \L:=\{f \in W^{1,2}_{\loc}(\X) \cap L^\infty(\X): |\Dif f| \le 1\ \text{a.e.} \} \comma 
\]
and the \emph{local Sobolev-to-Lipschitz property} means the property described in Proposition~\ref{pr:StL}: 
for any $f \in \L$, there exists a bounded $1$-Lipschitz function $\tilde{f}$ 
such that $f = \tilde{f}$ $\m$-a.e. 

Now, Theorem~\ref{t:m2} is proved in the same way as Proposition~\ref{pr:U} with $\bar\mssd_B$ and $\bar\mssd_\m(A, B)$ replaced by $\mssd_B$ and $\mssd(A, B)$, respectively, to obtain the upper bound 
\begin{align}\label{ineq:UBM}
\P_t(A, B) \le \sqrt{\mssm(A)}\sqrt{\mssm(B)} \exp \biggl( -\frac{\d(A, B)^2}{2t} \biggr) \fstop 
\end{align}
The infinitesimal strict convexity is used to have the integration by parts formula in the third equality in \eqref{e:KW0} in the generality of metric measure spaces.
We remark that $W^{1,2}(\X)$ is not necessarily a Hilbert space, therefore, the reflexivity is not automatic.
If $W^{1,2}(\X)$ is reflexive, then  the proof of Proposition~\ref{pr:dAB} works verbatim to construct $\bar\mssd_\mssm(A, B)$ based on $\mathbb L$ defined above in metric measure spaces, and we have
\begin{align*}%\label{ineq:UBM1}
\P_t(A, B) \le \sqrt{\mssm(A)}\sqrt{\mssm(B)} \exp \biggl( -\frac{\bar\d_\mssm(A, B)^2}{2t} \biggr) \fstop 
\end{align*} 
If, furthermore, the local Sobolev-to-Lipschitz property holds, 
the same proof of Lemma~\ref{lm:CDD} applies to obtain $\bar\mssd_\mssm(A, B)=\mssd(A, B)$ for any open sets $A, B \subset \X$. 

Concerning the lower estimate in the next section, 
our argument is applicable up to Subsection~\ref{ssc:limit}, 
whereas the use of a linearised heat semigroup in Lemma~\ref{l:215} prevents us from applying the same proof to the non-smooth case. 

\begin{rem}\label{rm:asymm}
The nonsmooth calculus in this subsection is not yet generalised to the asymmetric setting. 
We refer to \cite{KriZha22} for a related study of the curvature-dimension condition 
in asymmetric metric measure spaces. 
\end{rem}

\begin{rem}[The case without the infinitesimal strict convexity]\label{rm:stability}
The upper bound~\eqref{ineq:UBM} can be extended to some spaces without the infinitesimal strict convexity by using the stability under a perturbation of metric measure spaces.
For instance, take $\X=\R^n$, $\mssm$ as the Lebesgue measure, and $\mssd_\e(x,y)^2:=\|x-y\|_{p}^2+\e\|x-y\|_{2}^2$, where $\|x\|_p^p:=\sum_{i=1}^n|x_i|^p$ when $1 \le p <\infty$, and $\|x\|_\infty:=\max_{1 \le i \le n}|x_i|$.
As $(\R^n, \mssd_\e)$ is a normed space for every $\e>0$, the metric measure space $\X_\e:=(\R^n, \mssd_\e, \mssm)$ is a $\CD(0, n)$ space for every $\e>0$ (see Theorem in \cite[p.~908]{Vil09}).
Furthermore, $(\R^n, \mssd_\e)$ is infinitesimally strictly convex for every $\e>0$.
It is easy to see that $\X_\e$ converges to $\X=(\R^n, \d, \mssm)$ as $\e \to 0$ in the sense of the pointed measured Gromov--Hausdorff convergence, where $\d$ is induced from $\|\cdot\|_p$. 
Thus, by \cite[(1.5)]{GigMonSav15}, the corresponding heat semigroup $\T^\e_t f$ converges to $\T_t f$ strongly in $L^2(\R^n, \mssm)$ for every $f \in L^2(\R^n, \mssm)$ and $t>0$.
Noting that 
\[
\lim_{\e \to 0}\mssd_\e(A, B)=\inf_{\e>0}\inf_{x \in A,\, y \in B}\sqrt{\|x-y\|_p^2+\e\|x-y\|_2^2}  = \mssd(A, B) \comma
\]
we obtain 
\begin{align*}%\label{ineq:UBM}
\P_t(A, B) 
&= \int_{\R^n}\1_A \cdot \T_t \1_B \diff \mssm 
 = \lim_{\e \to 0}\int_{\R^n}\1_A \cdot \T^\e_t \1_B \diff \mssm 
\\
&\le \lim_{\e \to 0} \sqrt{\mssm(A)}\sqrt{\mssm(B)} \exp \biggl( -\frac{\d_\e(A, B)^2}{2t} \biggr) 
\\
&=\sqrt{\mssm(A)}\sqrt{\mssm(B)} \exp \biggl( -\frac{\d(A, B)^2}{2t} \biggr)
\fstop 
\end{align*}
This proves the upper bound in Theorem \ref{t:m2} for $(\R^n, \d, \mssm)$ ($1 \le p \le \infty$), although the infinitesimal strict convexity fails for $p=1,\infty$.
For general metric measure spaces, however, it is open whether we can remove the infinitesimal strict convexity in Theorem~\ref{t:m2}. 
\end{rem}

\section{Lower estimate} \label{sec:L}%%%%%
%%%%%%%%%%

This last section is devoted to the lower estimate of~\eqref{eq:dAB}: 
\begin{equation}\label{LB}
\liminf_{t \downarrow 0} t\log \P_t(A, B) \ge -\frac{\bar\d_\m(A,B)^2}{2} 
\end{equation}
for measurable sets $A,B \subset \M$ with $0<\m(A),\m(B)<\infty$.  Let $(\M,F,\m)$ be a complete Finsler manifold with $\Lambda_F <\infty$ with a measure $\m$. 
We will need the additional assumptions $\mathsf{C}_F,\mathsf{S}_F <\infty$ 
and $\m(\M)<\infty$ only in a later step in Subsections~\ref{ssc:Tauber}, \ref{ssc:last}. 

\subsection{Outline of the proof}\label{ssc:outline}%%%%%
%%%%%

Let us first remark that it is sufficient to show \eqref{LB} in the case of $\P_t(A,B)<1$. 
This is seen by replacing $\m$ with $c\m$ for $c>0$ such that $c\m(A)<1$. 
Indeed, $(\M,F,c\m)$ has the same heat flow as $(\M,F,\m)$, thereby $\P_t^{c\m}(A,B)=c\P_t^\m(A,B)$, 
and clearly $\bar\d_{c\m}(A,B)=\bar\d_\m(A,B)$. 
Thus, we have $\P_t^{c\m}(A,B) \le c\m(A) <1$, and \eqref{LB} for $c\m$ implies that for $\m$. 

Given $\e>0$, we set 
\[ D_\e := \{x \in A : \bar\d_B(x) \le \bar\d_\m(A,B) + \e\} \comma \] 
and observe $\m(D_\e)>0$ by Proposition~\ref{pr:dAB}. 
Then, recalling $\Phi_t =\Phi(-t\log \T_t \1_B)$ and noting $\log\P_t(A,B)<0$,
%for small $t>0$, 
we have
\begin{align} \label{LB1}
\limsup_{t \downarrow 0}\Phi\bigl(-t\log \P_t(A, B)\bigr) 
& \le \limsup_{t \downarrow 0}\Phi\bigl(-t\log \P_t(D_\e, B)\bigr)  
\\
&= \limsup_{t \downarrow 0}\Phi\biggl(-t\log \biggl[ \frac{1}{\mssm(D_\e)} \int_{D_\e} \T_t\1_B \diff \mssm \biggr] \biggr) \notag
\\
&\le  \limsup_{t \downarrow 0} \frac{1}{\mssm(D_\e)} \int_{D_\e}  \Phi(-t\log \T_t\1_B ) \diff \mssm \notag
\\
&=  \limsup_{t \downarrow 0} \frac{1}{\mssm(D_\e)} \int_{D_\e}  \Phi_t \diff \mssm \comma \notag
\end{align}
where the latter inequality is derived from Jensen's inequality 
since the function $s \mapsto \Phi(-t\log s)$ is convex on $[0,1]$ for sufficiently small $t>0$  
(see \cite[Lemma~2.1]{HinRam03}). 
Now, suppose that the following inequality holds:
\begin{align}\label{eq:G}
\limsup_{t \downarrow 0} \int_{D_\e}  \Phi_t \diff \mssm
 \le \int_{D_\e} \Phi\biggl(\frac{\bar\mssd_B^2}{2}\biggr) \diff \mssm \fstop
\end{align}
Then, we can continue the estimation in \eqref{LB1} as, with the help of \eqref{e:estimates} to see $\Phi(t) \le t$, 
\begin{align*}
 \le \frac{1}{\mssm(D_\e)} \int_{D_\e}  \Phi\biggl(\frac{\bar\mssd_B^2}{2}\biggr) \diff \mssm 
 \le \frac{1}{\mssm(D_\e)} \int_{D_\e}  \frac{\bar\mssd_B^2}{2} \diff \mssm 
 \le \frac{(\bar\d_\m(A,B)+\e)^2}{2} \fstop
\end{align*}
Since $\e>0$ was arbitrary, $\Phi$ is non-decreasing and $\Phi(t)=\Phi^K(t) \to t$ as $K \to \infty$, 
we conclude \eqref{LB}. 

The purpose of the rest of the section is to show \eqref{eq:G}, which completes the proof of~\eqref{LB}. 
In fact, at the end of the section, we will prove that equality holds (see \eqref{eq:HR2.22}). 

\subsection{Uniform bounds}\label{ssc:uniform}%%%%%
%%%%%

For $\delta \in (0,1)$ and $t>0$, we put 
\[ u_t^\delta := -t\log\bigl( (1-\delta)\T_t\1_B + \delta\bigr) \comma \qquad
 e^\delta_t:=-t\log \delta \fstop \]
Note that $u_t^{\delta} \ge 0$ since $0 \le \T_t \1_B \le 1$.
In addition, for $\varphi$, $\Phi$ and $\Psi$ in Subsection~\ref{ssc:cut}, we will abbreviate as 
$\varphi_t^\delta :=\varphi(u_t^\delta)$, $\Phi_t^\delta :=\Phi(u_t^\delta)$ and $\Psi_t^\delta :=\Psi(u_t^\delta)$.

We will denote by $(\cdot,\cdot)_{L^2}$ both the $L^2$-inner product 
and the paring between a one-form and a vector field with respect to $\m$. 
We also introduce the following notation: 
For a nonnegative function $\rho$, define 
\[ \E_{\rho}(f) :=\frac{1}{2} \int_\M F^*(\dif f)^2 \rho \diff\m \comma \qquad
 (f_1,f_2)_{L^2(\rho)} :=\int_\M f_1 f_2 \rho \diff\m \fstop \]

\begin{lem}\label{l:EQH}
For any $t>0$, we have $u_t^\delta-e^\delta_t \in H^1_0(\M)$ and, 
for every bounded nonnegative function $\rho \in H^1_0(\M) \cap L^1(\M)$, 
\[ \partial_t (\rho, \Phi^\delta_t)_{L^2} 
 = \frac{1}{t}(\rho, \Psi_t^\delta)_{L^2}
 +\frac{1}{2} \Bigl(\dif \bigl( \varphi'(u^{\delta}_t)^2\rho\bigr), \nabla (-u_t^\delta)\Bigr)_{L^2}
 -\frac{1}{t}\E_{\varphi'(u^{\delta}_t)^2\rho} (-u_t^\delta) \fstop \]
\end{lem}

\begin{proof}
We put $f:=(1-\delta)\1_B$ for brevity.
Observe that $u_t^\delta-e^\delta_t =\vartheta(\T_t f)$, 
where $\vartheta(s) :=-t\log((s+\delta)/\delta)$ is a Lipschitz function on $[0,\infty)$ with $\vartheta(0)=0$.
Thus, since $\T_t f \in H^1_0(\M)$, we have $u_t^\delta-e^\delta_t \in H^1_0(\M)$. 
We deduce from the chain rule for $\dif$ that 
\[%\label{eq:nabla}
\dif u_t^\delta =-\frac{t}{\T_t f+\delta} \dif \T_t f \comma \qquad 
 \nabla(-u_t^\delta) =\frac{t}{\T_t f+\delta} \nabla \T_t f \fstop 
\]
We remark that, to derive the latter equation, 
we needed $t/(\T_t f +\delta) >0$ because of the irreversibility of $F$. 
Combining this with the Leibniz rule for $\dif$, we have
\begin{align*}
\biggl(\dif \biggl[ \frac{\rho}{\T_t f +\delta} \biggr], \nabla \T_t f  \biggr)_{L^2}
&= \biggl( \frac{\dif\rho}{\T_t f+\delta}, \nabla \T_t f \biggr)_{L^2}
 -\biggl( \frac{\dif\T_t f}{(\T_t f+\delta)^2}, \nabla \T_t f \biggr)_{L^2(\rho)} \\
&= \frac{1}{t} \bigl( \dif\rho, \nabla(-u_t^\delta) \bigr)_{L^2}
 +\frac{1}{t^2}\bigl( \dif u_t^\delta, \nabla(-u_t^\delta) \bigr)_{L^2(\rho)} \fstop
\end{align*}
Hence, we obtain
\begin{align*}
(\rho, \partial_t u_t^\delta)_{L^2}
&= \frac{1}{t}(\rho, u_t^\delta)_{L^2} -\frac{t}{2} \biggl(\rho, \frac{\Delta \T_t f}{\T_t f +\delta}\biggr)_{L^2} 
\\
&= \frac{1}{t}(\rho, u_t^\delta)_{L^2} + \frac{t}{2}\biggl(\dif \biggl[ \frac{\rho}{\T_t f +\delta} \biggr], \nabla\T_t f\biggr)_{L^2}
\\
&=\frac{1}{t}(\rho, u_t^\delta)_{L^2} +\frac{1}{2} \bigl(\dif\rho, \nabla(-u_t^\delta) \bigr)_{L^2}
 - \frac{1}{t} \E_{\rho}(-u_t^\delta) \fstop
\end{align*}
Since $\partial_t (\rho, \Phi_t^\delta)_{L^2} =(\rho, \partial_t \Phi_t^\delta)_{L^2} =\bigl( \varphi'(u^{\delta}_t)^2 \rho, \partial_tu_t^\delta \bigr)_{L^2}$,
replacing $\rho$ with $\varphi'(u^{\delta}_t)^2 \rho$ in the above calculation completes the proof. 
\end{proof}

For a function $f:(0,\infty) \times \M \to \R$ (such as $(t,x) \mapsto \varphi_t^\delta(x)$), 
we will denote its time average by
\begin{align*} %\label{e:TA}
\bar{f}_t(x) :=\frac{1}{t} \int_0^t f_s(x) \diff s \comma 
\end{align*}
where $f_s(x):=f(s,x)$. 
The next lemma is a standard fact of the ($H^1_0(\M)$-valued) Bochner integral
(cf.~\cite[Lemma~2.5]{HinRam03}).

\begin{lem}\label{l:2.5}
Let $f:(0, T] \times \M \to \R$ be a bounded jointly measurable function 
such that $f_t \in H^1_0(\M)$ for all $t \in (0, T]$ and $\int_0^T \|f_t\|_{H^1}^2 \diff t <\infty$. 
Then, we have $\bar{f}_T \in H^1_0(\M)$ and 
\[ \E_{\rho}(\bar{f}_T) \le \frac{1}{T} \int_0^T \E_{\rho}(f_t) \diff t \]
for any bounded nonnegative function $\rho \in L^1_{\loc}(\M)$. 
\end{lem}

\begin{proof}
By hypothesis, $t \mapsto f_t \in H^1_0(\M)$ is Bochner integrable 
and we have $\bar{f}_T \in H^1_0(\M)$. 
Then, the claimed inequality is a consequence of the linearity of $\dif$ 
and Jensen's inequality for the convex function $(F^*)^2$: 
\begin{align*}
2\E_{\rho}(\bar{f}_T)
&= \int_\M F^* \biggl( \frac{1}{T} \int_0^T \dif f_t \diff t \biggr)^2 \rho \diff\m
 \le \frac{1}{T} \int_\M \int_0^T F^*(\dif f_t)^2 \rho \diff t \diff\m \\
&= \frac{2}{T} \int_0^T \E_{\rho}(f_t) \diff t \fstop
\end{align*}
\end{proof}

The next proposition is the goal of this subsection.

\begin{prop}\label{p:UBA}
For sufficiently small $T_0>0$, 
the families $\{\bar{\varphi}^\delta_t\chi_k\}_{0<t<T_0,\, 0<\delta<1}$ and $\{\bar{\Phi}^\delta_t\chi_k\}_{0<t<T_0,\, 0<\delta<1}$ are bounded in $H^1_0(\M)$ for every $k \in \N$.
\end{prop}

\begin{proof}
Since $\varphi$ and $\Phi$ are bounded,
it is straightforward that both families are bounded in $L^2(\M)$. 
Thus, we discuss only the bound for the energy. 
Put
\[ U_t^\delta := 2\E(-\varphi_t^\delta \chi_k) \comma \qquad a := \Lambda_F^2 \cdot \m(B_{k+1}) \comma \] 
for $\Lambda_F$ in \eqref{eq:rev}. 
Then, by \eqref{e:estimates} and the choice of $\chi_k$ as in Section~\ref{sc:max}, we have 
\begin{align}
U_t^\delta
&\le \bigl\| F^*\bigl( -\varphi'(u_t^{\delta}) \chi_k \, \dif u_t^{\delta} \bigr) 
 +F^*(-\varphi_t^{\delta} \, \dif \chi_k) \bigr\|_{L^2}^2
 \label{eq:U<}\\
&= \bigl\| \varphi'(u_t^\delta) \chi_k F^* (-\dif u_t^\delta)
 +\varphi_t^\delta F^* (-\dif \chi_k) \bigr\|_{L^2}^2
 \nonumber\\
&\le 4\E_{\varphi'(u^{\delta}_t)^2\chi_k^2}(-u_t^\delta) +2(KL)^2 \m(B_{k+1}) \fstop
 \nonumber
\end{align}
Letting $\rho=\chi_k^2$ in Lemma~\ref{l:EQH}, we find
\begin{align*}%\label{e:UBA1}
V_t^\delta 
&:= 2\E_{\varphi'(u^{\delta}_t)^2\chi_k^2}(-u_t^\delta) \\
&= -2t\partial_t(\Phi_t^\delta, \chi_k^2)_{L^2}
 + 2(\Psi_t^\delta, \chi_k^2)_{L^2}
 + t\Bigl( \dif\bigl( \varphi'(u^\delta_t)^2\chi_k^2 \bigr), \nabla(-u_t^\delta) \Bigr)_{L^2} \fstop
\end{align*}
Note that, in the RHS, $(\Psi_t^\delta, \chi_k^2)_{L^2} \le KL\m(B_{k+1})$ 
by \eqref{e:estimates} and the choice of $\chi_k$. 
For the third term, we deduce from the Leibniz rule for $\dif$ and \eqref{e:estimates} that 
\begin{align*}
&\Bigl(\dif\bigl(\varphi'(u^\delta_t)^2\chi_k^2\bigr), \nabla(-u_t^\delta)\Bigr)_{L^2}
\\
&=2\bigl(\dif\chi_k, \nabla(-u_t^\delta) \bigr)_{L^2(\varphi'(u_t^\delta)^2 \chi_k)}
 +2\bigl( \varphi''(u^\delta_t) \,\dif u_t^\delta, \nabla(-u_t^\delta) \bigr)_{L^2(\varphi'(u^\delta_t)\chi_k^2)}
\\
&\le 2\Bigl( F^*(\dif\chi_k), F\bigl( \nabla(-u_t^\delta) \bigr) \Bigr)_{L^2(\varphi'(u_t^\delta)^2 \chi_k)}
 +\frac{4C}{K} \E_{\varphi'(u^\delta_t)^2 \chi_k^2}(-u_t^\delta) \fstop
\end{align*}
Since 
$\dif(\varphi_t^\delta \chi_k) = \varphi'(u_t^\delta) \chi_k \, \dif u_t^\delta + \varphi_t^\delta \, \dif \chi_k$ 
by the chain rule for $\dif$, we also have 
\[ \varphi'(u_t^\delta) \chi_k F\bigl( \nabla (-u_t^\delta) \bigr)
 \le F\bigl( \nabla(-\varphi_t^\delta \chi_k) \bigr) + \varphi_t^\delta F(\nabla \chi_k) \fstop \]
Combining these inequalities with $2\E(\chi_k) \le 2\Lambda_F^2 \E(-\chi_k) \le a$ yields that 
\begin{align*}
V_t^\delta
&\le -2t\partial_t(\Phi_t^\delta, \chi_k^2)_{L^2} +2KLa +\frac{2Ct}{K} V_t^\delta \\
&\quad +2t\Bigl( F^*(\dif\chi_k), F\bigl( \nabla(-\varphi_t^\delta \chi_k) \bigr) \Bigr)_{L^2(\varphi'(u_t^\delta))}
 +2t\bigl( F^*(\dif\chi_k), F(\nabla \chi_k) \bigr)_{L^2(\varphi_t^\delta \varphi'(u_t^\delta))}
\\
&\le -2t\partial_t(\Phi_t^\delta, \chi_k^2)_{L^2} +2KLa +\frac{2Ct}{K} V_t^\delta
 +4t\sqrt{\E(\chi_k) \E(-\varphi_t^\delta \chi_k)} +4KLt \E(\chi_k)
\\
&\le -2t\partial_t(\Phi_t^\delta, \chi_k^2)_{L^2} +2KLa +\frac{2Ct}{K} V_t^\delta
 +2t\sqrt{a U_t^\delta}+2KLa t \fstop
\end{align*}
Hence, we obtain 
\begin{align}\label{eq:V<}
\biggl(1-\frac{2Ct}{K}\biggr)V_t^\delta 
&\le -2t\partial_t(\Phi_t^\delta, \chi_k^2)_{L^2} +\frac{U_t^\delta}{8} +8a t^2 +2KLa(t+1) \fstop
\end{align}

Now, put $T_0:=K/(4C)$ and observe that $V_t^\delta/2 \le (1-(2Ct)/K)V_t^\delta$ for $t \in (0, T_0]$. 
Thus, for $t \in (0, T_0]$, we find from \eqref{eq:U<} and \eqref{eq:V<} that 
\begin{align}
U_t^\delta 
&\le 2V_t^\delta +2K^2 L^2 a \label{eq:UV}\\
&\le -8t\partial_t(\Phi_t^\delta, \chi_k^2)_{L^2} + \frac{U_t^\delta}{2}
 +32a T_0^2 +8KLa(T_0 +1) +2K^2 L^2 a \fstop \nonumber
\end{align}
This implies 
\[ U_t^\delta \le -16t \partial_t(\Phi_t^\delta, \chi_k^2)_{L^2} +c \comma \]
where $c$ is a constant independent of $t \in (0,T_0]$ and $\delta \in (0,1)$. 
Hence, for $0<\e<t \le T_0$, 
\begin{align*}
\int_\e^t 2\E(-\varphi_s^\delta \chi_k) \diff s 
&= \int_\e^t U_s^\delta \diff s
\le -16\int_\e^t s \partial_s(\Phi_s^\delta, \chi_k^2)_{L^2} \diff s + c(t-\e)
\\
&= -16 \Big[ s(\Phi_s^\delta, \chi_k^2)_{L^2} \Big]_\e^t
 + 16\int_\e^t (\Phi_s^\delta, \chi_k^2)_{L^2} \diff s + c(t-\e) \fstop
\end{align*}
Letting $\e \to 0$ yields, since $0 \le \Phi \le KL$, 
\[ \int_0^t 2\E(-\varphi_s^\delta \chi_k) \diff s 
 \le 16KLat +ct \fstop \]
Therefore, we conclude 
\begin{align*} %\label{in:AV}
2\E(-\bar\varphi_t^\delta \chi_k) 
 \le \frac{1}{t} \int_0^t 2\E(-\varphi_s^\delta \chi_k) \diff s
 \le 16KLa +c \comma
\end{align*}
where the first inequality follows from Lemma~\ref{l:2.5}. 
This completes the proof of the boundedness of $\{ \bar\varphi_t^\delta \chi_k \}_{0<t<T_0,\, 0<\delta<1}$ 
in $H^1_0(\M)$. 

By a similar calculation to \eqref{eq:U<}, we find 
\[
 2\E(-\Phi_t^\delta \chi_k)
 \le \bigl\| \varphi'(u_t^\delta)^2 \chi_k F^* (-\dif u_t^\delta)
 +\Phi_t^\delta F^* (-\dif \chi_k) \bigr\|_{L^2}^2
 \le 2V_t^\delta +2(KL)^2 a \fstop
\]
Thus, we deduce from \eqref{eq:UV} that, for $t \in (0, T_0]$, 
\[
2\E(-\Phi_t^\delta \chi_k)
 \le -8t\partial_t(\Phi_t^\delta, \chi_k^2)_{L^2} +\frac{U_t^\delta}{2} +\frac{c}{2}
 \le -16t\partial_t(\Phi_t^\delta, \chi_k^2)_{L^2} +c \comma 
\]
which yields 
\begin{equation}\label{ineq:ADB}
\int_0^t 2\E(-\Phi_s^\delta \chi_k) \diff s \le 16KLat +ct
\end{equation}
and the boundedness of $\{ \bar\Phi_t^\delta \chi_k \}_{0<t<T_0,\, 0<\delta<1}$ in $H^1_0(\M)$ 
by the same reasoning as above. 
\end{proof}

\subsection{Limit functions}\label{ssc:limit}%%%%%
%%%%%

Thanks to Proposition~\ref{p:UBA} and the reflexivity of $H^1_0(\M)=H^1(\M)$, 
up to extracting a (non-relabelled) subsequence, we have 
\[%\label{e:wcp}
\bar\varphi_t^\delta \chi_k \xrightarrow{\delta \to 0,\, t \to 0} \exists \bar\varphi_k \comma
 \quad\text{weakly in $H^1_0(\M)$} \fstop
\]
Since $\bar\varphi_k = \bar\varphi_l$ on $B_k$ for $k \le l$, 
there exists a bounded nonnegative function $\bar\varphi_0 \in H^1_{\loc}(\M)$ 
such that $\bar\varphi_0=\bar\varphi_k$ on $B_k$ for every $k \in \N$. 
By the boundedness of $\Phi$ and $\Psi$, 
we may take a further (non-relabelled) subsequence such that, 
by passing $\delta \to 0$ and then $t \to 0$, 
\[%\label{e:wcp1}
 \Phi_t^{\delta} \xrightarrow{\delta \to 0,\, t \to 0} \exists \Phi_0 \comma \qquad 
 \bar\Phi_t^{\delta} \xrightarrow{\delta \to 0,\, t \to 0} \exists \bar\Phi_0 \comma \qquad 
 \bar\Psi_t^{\delta} \xrightarrow{\delta \to 0,\, t \to 0} \exists \bar\Psi_0
\]
for some nonnegative functions $\Phi_0, \bar\Phi_0, \bar\Psi_0 \in L^{\infty}(\M)$, 
both in the weak $L^2(\nu)$ sense for any finite measure $\nu$ mutually absolutely continuous with $\m$ 
and in the weak-star $L^\infty(\m)$ sense. 

Then the goal of this subsection is to prove the next proposition. 

\begin{prop}\label{p:LPE}
We have
\[ \bar\Phi_0= \Phi\biggl(\frac{\bar\mssd_B^2}{2}\biggr) \quad \text{a.e.} \]
In particular, $\bar\Phi_0$ is uniquely determined and is the weak limit of $\bar\Phi_t$ as $t \to 0$.
\end{prop}

Recall Proposition~\ref{pr:dAB} for $\bar\d_B$.
To this end, we first observe the following.

\begin{lem}\label{lm:LPE1}
We have
\[ \bar\Phi_0 \le \bar\varphi_0 \le \frac{\bar\d_B^2}{2} \quad \text{a.e.} \]
\end{lem}

\begin{proof} 
We will use the same symbols as in Proposition~\ref{p:UBA}. 
The former inequality~$\bar\Phi_0 \le \bar\varphi_0$ follows from $\Phi \le \varphi$. 
To show the latter inequality, take a bounded nonnegative function $\rho \in H^1_0(\M)$ 
with $\supp\, \rho \subset B_k$. 
Then, it follows from $\chi_k \equiv 1$ on $B_k$ and Lemma~\ref{l:EQH} that 
\begin{align*}
2\E_\rho (-\varphi_t^\delta\chi_k)
&= 2\E_{\varphi'(u^{\delta}_t)^2 \rho}(-u_t^\delta)
\\
&=-2t\partial_t(\Phi_t^\delta, \rho)_{L^2}
 +t\Bigl(\dif \bigl(\varphi'(u^{\delta}_t)^2 \rho \bigr), \nabla (-u_t^\delta)\Bigr)_{L^2} 
 +2(\Psi_t^\delta, \rho)_{L^2} \fstop
\end{align*}
By the Leibniz rule for $\dif$ and \eqref{e:estimates}, the second term can be bounded as 
\begin{align*}
&\Bigl(\dif \bigl(\varphi'(u^{\delta}_t)^2 \rho \bigr), \nabla (-u_t^\delta) \Bigr)_{L^2} 
\\
&= \bigl( \varphi'(u_t^\delta)^2 \, \dif \rho, \nabla (-u_t^\delta) \bigr)_{L^2} 
 +2\bigl( \varphi''(u_t^\delta) \, \dif u_t^\delta, \nabla (-u_t^\delta) \bigr)_{L^2(\varphi'(u^\delta_t)\rho)} 
\\
&\le \bigl( \dif \rho, \nabla (-\Phi_t^\delta) \bigr)_{L^2}
 + \frac{4C}{K} \E_{\varphi'(u_t^\delta)^2 \rho}(-u_t^\delta)
\\
&= \bigl( \dif \rho, \nabla(-\Phi_t^\delta\chi_k) \bigr)_{L^2} 
 + \frac{4C}{K} \E_\rho (-\varphi_t^\delta)
 \fstop
\end{align*}
Combining this with the above calculation and observing 
$\E_\rho (-\varphi_t^\delta) =\E_\rho (-\varphi_t^\delta \chi_k)$, 
we find 
\begin{align*} %\label{e:LPE1}
2\biggl( 1-\frac{2Ct}{K} \biggr) \E_\rho (-\varphi_t^\delta \chi_k)
 \le -2t\partial_t(\Phi_t^\delta, \rho)_{L^2} +t\bigl( \dif \rho, \nabla(-\Phi_t^\delta\chi_k) \bigr)_{L^2} 
 +2(\Psi_t^\delta, \rho)_{L^2} \fstop
\end{align*}
This implies, with the help of Fubini's theorem, 
\begin{align}\label{e:LPE2}
&2\biggl(1-\frac{2CT}{K}\biggr) \frac{1}{T} \int_0^T \E_\rho (-\varphi_t^\delta \chi_k) \diff t
\\
&\le -\frac{2}{T} \Big[ t(\Phi_t^\delta, \rho)_{L^2} \Big]_0^T +\frac{2}{T} \int_0^T (\Phi_t^\delta, \rho)_{L^2} \diff t
 +\frac{1}{T}\int_0^T t\bigl( \dif \rho, \nabla(-\Phi_t^\delta\chi_k) \bigr)_{L^2} \diff t 
 + 2(\bar\Psi_T^\delta, \rho)_{L^2} \notag
\\
&= -2(\Phi_T^\delta, \rho)_{L^2} + 2(\bar\Phi_T^\delta, \rho)_{L^2} 
 +\frac{1}{T}\int_0^T t \bigl( \dif \rho, \nabla(-\Phi_t^\delta\chi_k) \bigr)_{L^2} \diff t 
 + 2(\bar\Psi_T^\delta, \rho)_{L^2} \fstop\notag
\end{align}
Using the integration by parts, we can estimate the third term of the RHS as 
\begin{align*}
&\frac{1}{T}\int_0^T t \bigl( \dif \rho , \nabla(-\Phi_t^\delta\chi_k)\bigr)_{L^2} \diff t 
\\
&= \frac{1}{T} \Biggl[ t \int_0^t \bigl( \dif \rho, \nabla(-\Phi_s^\delta\chi_k)\bigr)_{L^2} \diff s\Biggr]_{t=0}^{t=T} 
 - \frac{1}{T}\int_0^T\int_0^t\bigl( \dif \rho, \nabla(-\Phi_s^\delta\chi_k)\bigr)_{L^2}\diff s \diff t 
 \\
& \le \int_0^T 2\sqrt{\E(\rho) \E(-\Phi_s^\delta\chi_k)} \diff s
 + \frac{1}{T} \int_0^T\int_0^t 2\sqrt{\E(-\rho) \E(-\Phi_s^\delta \chi_k)} \diff s \diff t 
\\
&\xrightarrow{\delta \to 0,\, T \to 0} 0 \comma
\end{align*}
where the convergence in the last line follows from the boundedness \eqref{ineq:ADB}. 

By the lower semi-continuity of $\E_\rho$, 
taking the limit of \eqref{e:LPE2} as $\delta \to 0,\, T \to 0$ along the subsequence 
taken in the beginning of Subsection~\ref{ssc:limit} yields 
\begin{align*}%\label{e:LPE3}
2\E_\rho \bigl( -\bar\varphi_0 \bigr)
\le \bigl( -2\Phi_0 + 2\bar\Phi_0 + 2\bar\Psi_0, \rho \bigr)_{L^2} 
\le \bigl( 4\bar\varphi_0, \rho \bigr)_{L^2} \comma
\end{align*}
where we used $\rho \ge 0$ 
and $0 \le \Psi \le \Phi \le \varphi$ in \eqref{e:estimates} in the latter inequality. 
This implies that, for any $\e>0$, 
\[
2\E_\rho \Bigl( -\sqrt{\smash[t]{\bar\varphi_0+\e} \rule{0pt}{1.8ex}} \Bigr) 
 = \frac{1}{2} \E_{\rho/(\bar\varphi_0+\e)} \bigl( -\bar\varphi_0 \bigr)
 \le \biggl(\bar\varphi_0, \frac{\rho}{\bar\varphi_0+\e} \biggr)_{L^2}
 \le \|\rho\|_{L^1} \fstop
\]
Letting $\e \to 0$ and noting that $\rho$ was arbitrary, 
we find that $f_0 :=\sqrt{\bar\varphi_0} \in H^1_{\loc}(\M)$ 
with $F^*(-\dif f_0) \le 1$ a.e. 
Moreover, a similar proof to \cite[Lemma~4.4]{AriHin05} shows that $\bar\varphi_0=0$ a.e.\ on $B$. 
Thus, we have $f_0 \in \L_{B,\sqrt{KL}}$, and then it follows from Proposition~\ref{pr:dAB} that 
\begin{align*} %\label{in:FD}
\bar\varphi_0 \le \bar\d_B^2 \quad \text{a.e.}
\end{align*}
Finally, one can improve $\bar\varphi_0 \le \bar\mssd_B^2$ to $\bar\varphi_0 \le \bar\mssd_B^2/2$ 
by the same argument as in \cite[Lemma~4.5]{AriHin05} (see also \cite[Lemma 2.12]{HinRam03})
thanks to Proposition~\ref{pr:U}. 
\end{proof}

We can also show a partial converse inequality as follows. 

\begin{lem}\label{l:IID}
We have 
\[ \bar\Phi_0 \ge \Phi \biggl( \frac{\bar\mssd_B^2}{2} \biggr) \quad \text{a.e.} \]
\end{lem}

\begin{proof}
We can follow the same lines as in \cite[Lemma~2.13]{HinRam03} by using Corollary~\ref{c:1}. 
\end{proof}

Now, we are ready to prove Proposition~\ref{p:LPE}. 

\begin{proof}[Proof of Proposition~\textnormal{\ref{p:LPE}}]
Thanks to Lemmas~\ref{lm:LPE1}, \ref{l:IID}, 
we can apply the argument in \cite[Lemma~2.14]{HinRam03} 
by replacing $\m$ with an equivalent finite measure $\nu$. 
\end{proof}

\subsection{Tauberian argument}\label{ssc:Tauber}%%%%%
%%%%%

By a Tauberian argument, we shall get rid of the time average from Proposition~\ref{p:LPE}. 
Here we need additional assumptions to make use of linearised heat semigroups. 
Note that, for the linearised gradient operator as in \eqref{eq:lin-gra}, we have 
\begin{equation}\label{eq:fh}
\dif h (\nabla f)
 = \sum_{i,j=1}^n g^*_{ij}(\dif f) \frac{\partial f}{\partial x^j} \frac{\partial h}{\partial x^i}
 = \dif f (\nabla^{\nabla f} h) \fstop
\end{equation}
Recall also that we set $u_t =\T_t \1_B$ and $\Phi_t =\Phi(-t \log \T_t \1_B)$ in Section \ref{sec:U}. 

\begin{lem}\label{l:215}
Assume $\mathsf{C}_F, \mathsf{S}_F<\infty$ and $\m(\M)<\infty$. 
Then, for every $\tau>0$ and measurable set $D \subset \M$ with $\m(D)>0$, we have 
\[ \lim_{t \downarrow 0} (\Phi_t, h^{\tau,\sigma}_{\tau-t})_{L^2}
 =(\bar\Phi_0, h^{\tau,\sigma}_{\tau})_{L^2} \comma \]
where $(h^{\tau,\sigma}_t)_{t \in [0,\tau]}$ is the solution to the linearised heat equation 
\begin{equation}\label{eq:linear}
\partial_t h^{\tau,\sigma}_t =\frac{1}{2} \Delta\!^{\nabla u_{\tau -t}} h^{\tau,\sigma}_t \comma
 \qquad h^{\tau,\sigma}_0 =\T_{\sigma} \1_D \fstop
\end{equation}
\end{lem}

\begin{rem}\label{rm:linear}
In \eqref{eq:linear}, to be precise, 
we choose a measurable one-parameter family $(V_t)_{t \ge 0}$ of nowhere vanishing vector fields 
with $V_t(x)=\nabla u_t(x)$ when $\dif u_t(x) \neq 0$, 
and replace $\nabla u_{\tau -t}$ with $V_{\tau -t}$. 
The unique existence of a solution $(h^{\tau,\sigma}_t)_{t \in [0,\tau]}$ is guaranteed 
in the same manner as that for \eqref{eq:lin-heat} (see \cite[Proposition 13.20]{Oht21}) 
by virtue of the hypothesis $\mathsf{C}_F, \mathsf{S}_F<\infty$. 
We remark that $t \mapsto h^{\tau,\sigma}_t$ is $L^2$-continuous on $[0,\tau]$, 
and $\int_\M h^{\tau,\sigma}_t \diff \m =\m(D)$ holds for all $t$ 
(since constant functions belong to $L^2(\M)$). 
\end{rem}

We also remark that choosing $\T_{\sigma} \1_D$ as the initial point is unessential; 
we may take any nonnegative function in $H^1_0(\M)$ converging to $\1_D$ 
(see the very last step in Subsection~\ref{ssc:last}). 

\begin{proof}
We follow the argument in \cite[Lemma 2.15]{HinRam03} (see also \cite[Lemma~4.6]{AriHin05}), 
however, a modification is needed because of the asymmetry \eqref{eq:asym}. 

Since $\tau,\sigma>0$ are fixed, we will denote $h^{\tau,\sigma}_t$ by $h_t$ for brevity. 
For $t \in (0,\tau)$, put $H(t) :=(\Phi_t, h_{\tau -t})_{L^2}$. 
For applying the Tauberian-type theorem in \cite[Lemma~3.11]{Ram01} 
which implies the claim $\lim_{t \downarrow 0} H(t) = (\bar\Phi_0, h_{\tau})_{L^2}$, 
it suffices to see the following: 
\begin{enumerate}[(a)]
\item\label{it:a} $T^{-1} \int_0^T H(t) \diff t \to (\bar\Phi_0, h_{\tau})_{L^2}$ as $T \to 0$;
\item\label{it:b} There exist $M, t_0>0$ such that $H(t)-H(s) \le M(t-s)/s$ for all $0<s<t \le t_0$.  
\end{enumerate}

The condition \ref{it:a} follows from
\begin{align*}
\biggl| \frac{1}{T}\int_0^T H(t) \diff t -(h_\tau, \bar\Phi_0)_{L^2} \biggr| 
&\le \frac{1}{T}\int_0^T \bigl| H(t) - (h_\tau, \Phi_t)_{L^2} \bigr| \diff t 
 + \bigl| (h_\tau, \bar\Phi_T-\bar\Phi_0)_{L^2} \bigr| 
\\
&\le \frac{KL}{T}\int_0^T \| h_{\tau-t} -h_\tau \|_{L^1} \diff t 
 + \bigl|( h_\tau, \bar\Phi_T-\bar\Phi_0)_{L^2} \bigr| 
\\
&\xrightarrow{T \to 0} 0 \fstop
\end{align*}

To see \ref{it:b}, we observe from Lemma~\ref{l:EQH} and \eqref{eq:linear} that
\begin{align*}
\Bigl[ (h_{\tau-r}, \Phi_r^\delta)_{L^2} \Bigr]_{r=s}^{r=t} 
&= \int_s^t( h_{\tau-r}, \partial_r\Phi_r^\delta)_{L^2}\diff r
 + \int_s^t (\partial_r h_{\tau-r}, \Phi_r^\delta)_{L^2}\diff r
\\
&= \int_s^t\frac{1}{r} ( h_{\tau-r}, \Psi_r^\delta)_{L^2}\diff r 
 + \int_s^t \frac{1}{2} \Bigl(\dif \bigl( \varphi'(u_r^\delta)^2 h_{\tau-r} \bigr), \nabla(-u_r^\delta) \Bigr)_{L^2}\diff r
\\
&\quad -\int_s^t \frac{1}{r} \E_{\varphi'(u^{\delta}_r)^2 h_{\tau-r}} (-u_r^\delta) \diff r 
 - \int_s^t \frac{1}{2} (\Delta\!^{\nabla u_r} h_{\tau -r}, \Phi_r^\delta)_{L^2} \diff r 
\\
&=:I_1 +I_2 -I_3 -I_4 \fstop
\end{align*}
Note first that $I_1$ can be estimated as 
\[
I_1 \le \int_s^t \frac{1}{r} KL\, \m(D) \diff r \le KL\, \m(D) \frac{t-s}{s} \fstop
\]
Next, since 
\[ \dif \Phi^{\delta}_r =\varphi'(u^\delta_r)^2 \,\dif u^\delta_r
 = -\varphi'(u^\delta_r)^2 \frac{t(1-\delta)}{(1-\delta) u_r +\delta} \cdot \dif u_r \]
implies $g^*_{ij}(\dif u_r) =g^*_{ij}(-\dif \Phi^\delta_r)$, we deduce from \eqref{eq:fh} that 
\begin{align*}
2I_4 &= \int_s^t \bigl( \dif (-\Phi_r^\delta), \nabla^{\nabla u_r} h_{\tau -r} \bigr)_{L^2} \diff r
 = \int_s^t \bigl( \dif h_{\tau -r}, \nabla (-\Phi^\delta_r) \bigr)_{L^2} \diff r 
\\
&= \int_s^t \bigl( \dif h_{\tau -r}, \nabla (-u^\delta_r) \bigr)_{L^2(\varphi'(u^\delta_r)^2)} \diff r 
\end{align*}
(we remark that the integrands in the above calculation vanish on the set $\{\dif u_r =0\}$). 
Then, using the Leibniz rule for $\dif$ and \eqref{e:estimates}, we find 
\begin{align*}
I_2 
&= \int_s^t \frac{1}{2}
 \bigl(\dif h_{\tau-r}, \nabla(-u_r^\delta) \bigr)_{L^2(\varphi'(u_r^\delta)^2)} \diff r
 +\int_s^t \bigl( \varphi''(u_r^\delta) \,\dif u^\delta_r, \nabla(-u^\delta_r) \bigr)_{L^2(\varphi'(u^{\delta}_r) h_{\tau-r})} \diff r 
\\
& \le I_4 +\frac{2C}{K}\int_s^t  \E_{\varphi'(u^{\delta}_r)^2 h_{\tau-r}} (-u_r^\delta) \diff r 
 \fstop
\end{align*}
Setting $t \le t_0 := K/(2C)$, we have $I_2 \le I_4 +I_3$. 
Therefore, we obtain
\[ \Bigl[ (h_{\tau-r}, \Phi_r^\delta)_{L^2} \Bigr]_{r=s}^{r=t}
 \le I_1  \le KL\,\m(D) \frac{t-s}{s} \quad\, \text{for all}\ 0<s<t \le t_0 \fstop \]
Letting $\delta \to 0$ completes the proof. 
\end{proof}

\subsection{Proof of \eqref{eq:G}}\label{ssc:last}%%%%%
%%%%%

We are now in a position to prove \eqref{eq:G}. 
We shall in fact show that 
\begin{equation}\label{eq:HR2.22}
\lim_{t \downarrow 0} \int_D \Phi_t \diff \m 
 = \int_D \Phi\biggl(\frac{\bar\mssd_B^2}{2}\biggr) \diff \m
\end{equation}
for every measurable set $D \subset \M$ with $\m(D)>0$, 
namely $\Phi(\bar\d_B^2/2)$ is the weak $L^2$-limit of $\Phi_t$.
For $\sigma>0$ and $\tau>t>0$, we decompose as 
\[ \int_D \Phi_t \diff \mssm
 =(\Phi_t, h^{\tau,\sigma}_{\tau-t})_{L^2}+(\Phi_t, \1_D -h^{\tau,\sigma}_{\tau-t})_{L^2} \fstop \]
Taking the limit as $t \to 0$, we find from Lemma~\ref{l:215}, Proposition~\ref{p:LPE} and $0 \le \Phi \le KL$ that
\begin{equation}\label{eq:h_t}
\Biggl| \lim_{t \downarrow 0} \int_D \Phi_t \diff \mssm
 - \biggl( \Phi\biggl( \frac{\bar{\d}_B^2}{2} \biggr), h^{\tau,\sigma}_{\tau} \biggr)_{L^2} \Biggr|
 \le KL \| \1_D -h^{\tau,\sigma}_{\tau} \|_{L^1} \fstop
\end{equation}

We next consider the limit as $\tau \to 0$.
Put $w_t:=\T_t \1_D$ and observe that, for $t \in (0,\tau)$,
\begin{align*}
&\frac{\diff}{\diff t} \Bigl[ \| w_{\sigma+t} -h^{\tau,\sigma}_t \|_{L^2}^2 \Bigr] \\
&= \int_\M ( w_{\sigma+t} -h^{\tau,\sigma}_t) 
 \bigl( \Delta w_{\sigma +t} -\Delta\!^{\nabla u_{\tau-t}} h^{\tau,\sigma}_t \bigr) \diff\m \\
&= -\int_\M \dif(w_{\sigma+t} -h^{\tau,\sigma}_t)
 \bigl( \nabla w_{\sigma +t} -\nabla^{\nabla u_{\tau -t}} h^{\tau,\sigma}_t \bigr) \diff\m \\
&= -2\E(w_{\sigma +t}) -2\E^{\nabla u_{\tau -t}}(h^{\tau,\sigma}_t)
 +\int_\M \bigl\{ \dif h^{\tau,\sigma}_t (\nabla w_{\sigma +t})
 +\dif w_{\sigma +t} \bigl( \nabla^{\nabla u_{\tau -t}}h^{\tau,\sigma}_t \bigr) \bigr\} \diff\m \\
&\le -2\E(w_{\sigma +t}) -2\E^{\nabla u_{\tau -t}}(h^{\tau,\sigma}_t)
 +\int_\M F^*(\dif w_{\sigma +t}) \bigl\{ F^*(\dif h^{\tau,\sigma}_t) 
 +F\bigl( \nabla^{\nabla u_{\tau -t}}h^{\tau,\sigma}_t \bigr) \bigr\} \diff\m \comma 
\end{align*}
where we defined (recall \eqref{eq:dual} for $g^*_{\dif u}$) 
\[ \E^{\nabla u}(h) :=\frac{1}{2} \int_\M g^*_{\dif u} (\dif h,\dif h) \diff\m
 =\frac{1}{2} \int_\M \dif h(\nabla^{\nabla u} h) \diff\m \fstop \]
Now, it follows from \eqref{eq:CSdual} that 
\[ F^*(\dif h^{\tau,\sigma}_t)^2
 \le \mathsf{S}_F \cdot g^*_{\dif u_{\tau -t}}(\dif h^{\tau,\sigma}_t, \dif h^{\tau,\sigma}_t) \comma \]
and similarly, by \eqref{eq:CS}, 
\[ F\bigl( \nabla^{\nabla u_{\tau -t}}h^{\tau,\sigma}_t \bigr)^2
 \le \mathsf{C}_F \cdot g_{\nabla u_{\tau -t}}
 \bigl( \nabla^{\nabla u_{\tau -t}}h^{\tau,\sigma}_t,\nabla^{\nabla u_{\tau -t}}h^{\tau,\sigma}_t \bigr)
 = \mathsf{C}_F \cdot g^*_{\dif u_{\tau -t}}(\dif h^{\tau,\sigma}_t, \dif h^{\tau,\sigma}_t) \fstop \]
Therefore, 
\begin{align*}
&\frac{\diff}{\diff t} \Bigl[ \| w_{\sigma+t} -h^{\tau,\sigma}_t \|_{L^2}^2 \Bigr] \\
&\le -2\E(w_{\sigma +t}) -2 \E^{\nabla u_{\tau -t}}(h^{\tau,\sigma}_t) \\
&\quad +\int_\M \bigl( \sqrt{\mathsf{S}_F} +\sqrt{\mathsf{C}_F} \bigr) F^*(\dif w_{\sigma +t})
 \cdot \sqrt{g^*_{\dif u_{\tau -t}}(\dif h^{\tau,\sigma}_t, \dif h^{\tau,\sigma}_t)} \diff\m \\
&\le -2\E(w_{\sigma +t}) -2 \E^{\nabla u_{\tau -t}}(h^{\tau,\sigma}_t) \\
&\quad +\int_\M \biggl\{ \frac{1}{4} \bigl( \sqrt{\mathsf{S}_F} +\sqrt{\mathsf{C}_F} \bigr)^2 F^*(\dif w_{\sigma +t})^2
 +g^*_{\dif u_{\tau -t}}(\dif h^{\tau,\sigma}_t, \dif h^{\tau,\sigma}_t) \biggr\} \diff\m \\
&= \biggl( \frac{(\sqrt{\mathsf{S}_F}+\sqrt{\mathsf{C}_F})^2}{2} -2 \biggr) \E(w_{\sigma +t}) \fstop
\end{align*}
Since $h^{\tau,\sigma}_0=w_{\sigma}$, integrating the above inequality in $t \in (0,\tau)$ yields
\[ \| w_{\sigma+\tau} -h^{\tau,\sigma}_{\tau} \|_{L^2}^2
 \le \biggl( \frac{(\sqrt{\mathsf{S}_F}+\sqrt{\mathsf{C}_F})^2}{2} -2 \biggr)
 \int_0^\tau \E(w_{\sigma +t}) \diff t \fstop \]
Plugging
\[ \frac{\diff}{\diff t} \Bigl[ \|w_{\sigma +t}\|_{L^2}^2 \Bigr]
 = \int_\M w_{\sigma +t} \Delta w_{\sigma +t} \diff\m
 = -2\E(w_{\sigma +t}) \]
into the RHS, we obtain 
\[ \| w_{\sigma+\tau} -h^{\tau,\sigma}_{\tau} \|_{L^2}^2
 \le \biggl( \frac{(\sqrt{\mathsf{S}_F}+\sqrt{\mathsf{C}_F})^2}{4} -1 \biggr)
 \bigl( \|w_{\sigma}\|_{L^2}^2 -\|w_{\sigma +\tau}\|_{L^2}^2 \bigr) \fstop \]
Hence, $h^{\tau,\sigma}_\tau \to w_{\sigma}$ in $L^2$ as $\tau \to 0$, and then \eqref{eq:h_t} shows 
\[ \Biggl| \lim_{t \downarrow 0} \int_D \Phi_t \diff \mssm
 - \biggl( \Phi\biggl( \frac{\bar{\d}_B^2}{2} \biggr), w_{\sigma} \biggr)_{L^2} \Biggr|
 \le KL \| \1_D -w_{\sigma} \|_{L^1} \fstop \]
Finally, as $\sigma \to 0$, we have $\|\1_D -w_{\sigma}\|_{L^1} \to 0$ 
since $w_{\sigma} \to \1_D$ in $L^2$ and $\m(\M)<\infty$. 
This completes the proof of \eqref{eq:HR2.22} and hence the lower estimate \eqref{LB}. 

\subsection{Further problems}\label{ssc:outro}%%%%%
%%%%%

We conclude with some further problems. 

First of all, in Theorem~\ref{t:m1}, we used the assumption $\m(\M)<\infty$ 
only for deducing the $L^1$-convergence from the $L^2$-convergence. 
We expect that a finer analysis of (non-linear and linearised) heat semigroups could remove it. 
Such an analysis will be helpful also for the further study of geometric analysis on noncompact Finsler manifolds, 
where some results are known only under seemingly artificial assumptions; 
we refer to \cite{Oht22} for gradient estimates and an isoperimetric inequality, 
and to \cite{Mai22} for a rigidity problem of the spectral gap. 

In the non-smooth setting in Theorem~\ref{t:m2}, 
the lower bound estimate is an intriguing open problem. 
The main issue is whether we can  avoid using a linearised heat semigroup in Lemma~\ref{l:215}.

\bibliographystyle{abbrv}
\bibliography{MasterBib.bib}
\end{document}